\documentclass[letterpaper]{amsart}
\usepackage{amssymb,comment,bbm,amsmath,amsfonts,amsthm}
\usepackage{appendix}

\def\restriction#1#2{\mathchoice
              {\setbox1\hbox{${\displaystyle #1}_{\scriptstyle #2}$}
              \restrictionaux{#1}{#2}}
              {\setbox1\hbox{${\textstyle #1}_{\scriptstyle #2}$}
              \restrictionaux{#1}{#2}}
              {\setbox1\hbox{${\scriptstyle #1}_{\scriptscriptstyle #2}$}
              \restrictionaux{#1}{#2}}
              {\setbox1\hbox{${\scriptscriptstyle #1}_{\scriptscriptstyle #2}$}
              \restrictionaux{#1}{#2}}}
\def\restrictionaux#1#2{{#1\,\smash{\vrule height .8\ht1 depth .85\dp1}}_{\,#2}} 
\usepackage[T1]{fontenc}
\usepackage{graphicx}
\usepackage{lmodern}
\usepackage{stmaryrd}
\usepackage{fancyhdr}
\usepackage{dsfont}
\usepackage[all]{xy}
    \usepackage{hyperref}
    \hypersetup{colorlinks=true,linkcolor=blue}
\usepackage{framed}
\usepackage[many]{tcolorbox}
\usepackage{dirtytalk}

\newtheorem{lemme}{Lemma}[section]
\newtheorem{prop}{Proposition}[section]
\newtheorem{theoreme}{Theorem}[section]
\newtheorem{Prop}[theoreme]{Proposition}
\newtheorem{definition}[theoreme]{Definition}
\newtheorem{lem}[theoreme]{Lemma}
\newtheorem{Rem}[theoreme]{Remark}
\newtheorem{coro}[theoreme]{Corollary}
\newtheorem{cor}[prop]{Corollary}
\newtheorem{Theor}[theoreme]{Theorem}
\newtheorem{Theo}[prop]{Theorem}
\newtheorem{Rq}[prop]{Remark}
\newtheorem{Lem}[prop]{Lemma}
\newtheorem{Theorem}[lemme]{Theorem}
\newtheorem{Rk}[lemme]{Remark}

\title{Steklov isospectrality of conformal metrics}
\author{Benjamin Florentin}
\address{Université de Lorraine, CNRS, Inria, IECL, F-54000 Nancy, France}
\email{benjamin.florentin@univ-lorraine.fr}
\begin{document}
\begin{abstract}
The Steklov spectrum of a smooth compact Riemannian manifold $(M,g)$ with boundary is the set of eigenvalues counted with multiplicities of its Dirichlet-to-Neumann map (DN map).\\This article is devoted to the Steklov spectral inverse problem of recovering the metric $g$, up to natural gauge invariance, from its Steklov spectrum. Positive results are established in dimension $n\geq 3$ for conformal metrics under the assumption that the geodesic flow on the boundary is Anosov with simple length spectrum. The paper combines wave trace formula techniques with the injectivity of the geodesic X-ray transform for functions on closed Anosov manifolds. It is shown that knowledge of the Steklov spectrum determines the jet at the boundary of the underlying Riemannian metric within its conformal class. In this particular context, this parallels the well-known results of the Calder{\'o}n problem, where we are given the entire Dirichlet-to-Neumann map instead. As a simple corollary, assuming real-analyticity of the conformal factor, Steklov isospectral metrics must coincide. \\Using similar arguments,  we are also able to prove under the same assumption of hyperbolicity of the geodesic flow on the boundary, that generically any smooth potential $q$ can be recovered from the Steklov spectrum, in the sense that its jet at the boundary is determined by the spectrum of the DN map for the Schrödinger operator with potential $q$. Consequently, in this case, two analytic Steklov isospectral potentials must be equal.   
\end{abstract}
\maketitle
\tableofcontents

\section{Introduction}\label{Intro}

\subsection{Inverse problems for the Steklov spectrum}\label{Subsection I.1}

Let $(M,g)$ be a smooth compact Riemannian manifold of dimension $n\geqslant 2$ with smooth boundary $\partial M$. \\Recall that the \textit{Steklov eigenvalue problem} consists in finding scalars $\sigma$ for which the system 

$$\left\lbrace
\begin{aligned}
-\Delta_{g}u &= 0\hspace{0.3cm}\text{in $M$}\\
\partial_{\nu}u &= \sigma u\hspace{0.2cm}\text{on $\partial M$} \\
\end{aligned}\right.$$

\medskip\noindent admits a non-trivial solution in $H^{1}(M)$; where $\Delta_{g}$ is the Laplace-Beltrami operator on $(M,g)$ and $\partial_{\nu}$ denotes the Riemannian outward normal derivative along the boundary.
This famous problem, introduced at the beginning of the 20th century by the mathematician V.A. Steklov (see \cite{KKStek} for background), lies at the interface between spectral geometry and mathematical physics. We refer the reader to the excellent surveys \cite{GiPo} and \cite{ColGiGordSher} for details. 
\\Equivalently, one can consider the eigenvalues of the Dirichlet-to-Neumann map (DN map) $\Lambda_{g}$ defined by 
$$\begin{array}{l|rcl}
\Lambda_{g} : & H^{1/2}(\partial M) & \longrightarrow & H^{-1/2}(\partial M) \\
			& f & \longmapsto & \restriction{\partial_{\nu}u}{\partial M} \end{array}\hspace{0.1cm},$$ 
where $u$ is the unique solution in $H^{1}(M)$ of the Dirichlet problem

$$\left\lbrace
\begin{aligned}
-\Delta_{g}u &= 0 \\
\restriction{u}{\partial M} &= f \\
\end{aligned}\right.\hspace{0.3cm}\cdot$$

\medskip\noindent The DN map is a positive self-adjoint elliptic and classical (in the sense of \eqref{eq1}) pseudodifferential operator of order $1$ on the closed manifold $\partial M$. (see \cite{Tay}) \\In particular, its spectrum is discrete and is given by a sequence of eigenvalues 

$$0=\sigma_{0}<\sigma_{1}\leq\sigma_{2}\leq...\rightarrow \infty\hspace{0.1cm}.$$

\medskip\noindent The multiset
$$\mathrm{spec}(\Lambda_{g}):=\left(\sigma_{k}\right)_{k\geqslant 0}$$ 

\medskip\noindent defines the so-called Steklov spectrum of the Riemannian manifold $(M,g)$.  
\\In this context, following the celebrated paper \say{Can you hear the shape of a drum ?} (cf. \cite{Kac}), one may also ask whether the Steklov spectrum determine its underlying Riemannian structure. In other words, does the Steklov spectrum determine the metric up to a natural gauge invariance ? This problem is called the Steklov inverse problem, we refer once again to the surveys \cite{GiPo} and \cite{ColGiGordSher}.
\\It is well known that the answer is generally negative; in fact, there are many non-isometric Steklov isospectral (with the same Steklov spectrum) manifolds; the constructions are mainly adaptations of two general methods originally introduced to construct Laplace isospectral manifolds : Sunada's technique and the torus action technique. We refer to \cite{GordWebb} and references therein for details.
\vspace{0.2cm}

In this paper, we will study isospectral sets of conformal metrics, that is, metrics of the form $h=cg$ where $c:M\rightarrow \mathbb{R}$ is a smooth positive function in $M$. In that case, the Steklov spectral inverse problem boils down to showing that the map
$$h\mapsto \mathrm{spec}(\Lambda_{h})$$ 
is injective on the space of metrics conformal to $g$ on $M$. Of course the problem is difficult since this map is highly non-linear and it is also possible to find many counterexamples, particularly when considering warped product manifolds (cf. \cite{Gend} and references therein). Let us point out that the adaptation of Sunada's method cited above (and more precisely of \cite[Theorem 1]{BrPeYa}) to produce counterexamples is probably easy to adapt in dimension greater than $2$.
The aim of this paper is to establish 'positive' results, i.e. to exhibit a certain class of metrics for which it is possible to recover the geometry from the Steklov spectrum. The natural approach is to look for \textit{spectral invariants}, i.e geometric quantities that can be determined by the spectrum. These are usually provided by considering functions constructed from eigenvalues ; the first example is given by the well-known Weyl's asymptotic on the spectral counting function, which shows that the volume of the boundary can be recovered from the Steklov spectrum (see \textbf{Proposition \ref{Prop II.1}}).

\vspace{0.2cm}

More generally, a complete set of spectral invariants can be defined by considering 'special' functions, in particular, the trace of the heat kernel (for the Steklov spectrum, see for example \cite{PolSher}) or the trace of the wave operator, recalled in Section \ref{Subsection II.2}. We refer the reader to the excellent survey \cite{Zeld} for further discussions on this topic. On a closed Riemannian manifold, the spectral invariants arising from the trace of the wave operator associated with any self-adjoint positive elliptic pseudodifferential operator are given by the trace formula of Duistermaat and Guillemin \cite[Theorem 4.5]{GD}, providing a careful analysis of its singularities as a distribution in time. In the case of the Laplace-Beltrami operator, this formula was first established independently by Y.Colin de Verdière \cite{Cdv1}, \cite{Cdv2} and then by Chazarain \cite{Cha}.
Note that the exploitation of these trace formulas requires important assumptions about the dynamics of the Hamiltonian flow associated with the principal symbol of the operator under consideration, e.g. geodesic flows in the case of the Laplacian and the DN map. For geodesic flows, a particularly appropriate framework is provided by the negative curvature assumption or, more generally, by the Anosov property (see Section \ref{Section II}).
\vspace{0.2cm}

One of the most important examples of application is given by Guillemin and Kazhdan \cite{GK}, who were able to show the absence of non-trivial Laplace isospectral deformations (i.e one-parameter families $(g_{s})$ of smooth metrics preserving the Laplacian spectrum) on surfaces of negative curvature, and then in any dimension in \cite{GK2} (see also \cite{CrSh}). \\Recently, Paternain, Salo and Uhlmann \cite{PatSalUhl2} were able to extend these results to the case of Anosov surfaces.
The restriction of this spectral inverse problem to isospectral deformations is actually a linearization and leads more precisely to a linear tomography problem. The latter asks whether it is possible to recover a tensor from its integrals along closed geodesics, which corresponds to an injectivity property of the so-called X-ray transform (see Section \ref{Subsection II.1}).  \\Inspired by this approach, this paper also considers isospectral deformations for the Steklov spectrum :

\begin{definition}\label{Def I.1}
A deformation of $g$ is a smooth one-parameter family of (smooth) Riemannian metrics $\left(g_{s}\right)_{0\leqslant s\leqslant \varepsilon}$ on $M$ such that $g_{0}=g$. 
It is called Steklov isospectral if the spectrum of $\Lambda_{g_{s}}$ is preserved : 
$$\forall s\in \left[ 0,\varepsilon\right]\hspace{0.1cm},\hspace{0.1cm}\mathrm{spec}\left( \Lambda_{g_{s}}\right)=\mathrm{spec}\left( \Lambda_{g}\right)\hspace{0.1cm}.$$
\end{definition}
\medskip\noindent It is possible to establish, under natural assumptions, that isospectral deformations of conformal metrics are trivial in the sense that each metric automatically coincides at the boundary (see \textbf{Proposition \ref{Prop III.1}}). But, on the other hand, the trace formula techniques do not allow us to prove such a result for the non-linear problem (i.e. considering two conformal metrics). However, we point out that this can be shown under more restrictive assumptions on the geometry, exploiting the volume of the boundary as a spectral invariant, even if this is not very satisfactory in practice because very little spectral information is used, and consequently the geometric assumptions required are relatively strong. We refer the reader to Appendix \ref{App A}.

\vspace{0.2cm}
More generally, one cannot expect to obtain positive results on the general Steklov inverse problem without sufficiently strong assumptions, especially since, by comparison, we start with even less information than in the famous anisotropic Calder{\'o}n inverse problem, which consists in recovering the metric $g$ from knowledge of the DN map $\Lambda_{g}$ and is still far from being fully understood despite a vast literature on the subject. The reader may wish to consult, for example, the survey \cite{Daud1} and all references therein. Regarding this problem, a standard step is to show a boundary determination result, namely that the DN map determines, in boundary normal coordinates, all the normal derivatives of the metric tensor $g$ at the boundary at any order $j\geqslant 0$ (and thus the full Taylor series of $g$ at the boundary in these coordinates).
We would like to mention in particular the important papers \cite{LU}, \cite{LasUhl}, \cite{LasTayUhl}, \cite{DDSF} and the recent \cite{Cek}, \cite{LasLiiSal}. 

\vspace{0.2cm}
In this work, we establish similar boundary determination results for the Steklov spectrum. In particular, combining the boundary determination mentioned above, the use of wave trace invariants and the injectivity of the X-ray transform, we show that, under natural assumptions on the geodesic flow of $\restriction{g}{\partial M}$, the Steklov spectrum determines all normal covariant derivatives at the boundary of the metric $g$ within a fixed conformal class. Note that we are working here in dimension $n\geqslant 3$, the case of dimension $2$ is specific due to the well-known conformal invariance (see \textbf{Remark \ref{Rk III.9}}) ; we will simply observe that the Steklov inverse problem considered in this paper then reduces to showing that two conformal and isospectral metrics coincide at the boundary which, in the non-linear case, is a priori as difficult (and not true in general) as in higher dimensions. This point is actually the most troublesome one in the boundary determination, since wave invariants provide no useful information and then, as mentioned above, we are only able to obtain results under much more restrictive geometric assumptions by using more global spectral invariants, in this case the volume of the boundary (see Appendix \ref{App A}).

\vspace{0.2cm}
An important idea in this paper is to exploit the subprincipal symbol of the DN map (see the beginning of Section \ref{Section IV} and the references therein for a reminder of this notion), which unlike the case of the Laplace-Beltrami operator, is non-zero. This makes it possible to obtain more information from the principal wave invariants given by \textbf{Theorem \ref{Th II.3}}, under the additional assumption that the length spectrum of $\partial M$ is simple ; recall that a metric is said to have \textit{simple length spectrum} if all its closed geodesics have different lengths. It is well-known that this assumption is satisfied by generic Riemannian metrics even in the non-Anosov case (cf. \cite{Ab}, \cite[Lemma 4.4.3]{Kli} or \cite{An}).
\\To the author's knowledge, this is the first example of a spectral inverse problem exploiting this observation on the subprincipal symbol. A similar yet interesting example, involving the recovery of a potential $q$ from the spectrum of the associated DN map $\Lambda_{g,q}$, will be discussed in part \ref{Section V}. \\It should also be noted that the literature establishing positive results for the Steklov spectral inverse problem is extremely limited, and this paper provides, to our knowledge, the first rigidity results for such a broad class of geometries.

\vspace{0.2cm}
To conclude, note that another version of the wave trace formula in the context of isospectral connections is used in a similar way in the recent \cite{CekLef} to show an injectivity result for the \textit{spectrum map} (see \cite[Theorem 1.1]{CekLef}) solving in particular positively, on closed negatively curved manifolds with simple length spectrum, the inverse spectral problem for the magnetic Laplacian, namely recovering the magnetic potential from the spectrum.
On the Steklov spectral inverse problem, let us also mention the work of \cite{JolSha} on compact and simply connected surfaces with boundary, \cite{Gend} and \cite{Daud} within a conformal class of certain warped product manifolds or the very recent \cite{CekSif} for the magnetic Dirichlet-to-Neumann map on surfaces. Finally, for the Steklov spectrum in the context of conformal metrics, \cite{Jam} shows that on any compact manifold of dimension $n\geqslant 3$ with boundary, it is possible to prescribe any finite part of the Steklov spectrum within a given conformal class.

\vspace{0.2cm}
Throughout the article, for any Riemannian manifold $(M,g)$ with boundary, we will note $g^{\circ}:=\restriction{g}{\partial M}$ the induced metric on the closed manifold $\partial M$.

\subsection{Main results}\label{Subsection I.2}

This section gathers all the main spectral rigidity results obtained in the paper. Turning first to the non-linear problem in a conformal class, we need to require equality of the metrics at the boundary, information that we are unable to obtain directly via the wave trace invariants and which is actually not verified in general. The following result is essentially established in Section \ref{Section IV} and more particularly in \textbf{Theorem \ref{Th IV.6}}.

\begin{theoreme}\label{Th I.1}
    Let $(M,g)$ be a smooth compact Riemannian manifold of dimension $n\geq 3$ with boundary. Assume that $(\partial M,g^{\circ})$ is Anosov with simple length spectrum.
    \\Let $c:M\rightarrow \mathbb{R}^{+}_{*}$ be a smooth function satisfying $\restriction{c}{\partial M} \equiv 1$. If the Steklov spectra agree $$\mathrm{spec}(\Lambda_{cg})=\mathrm{spec}(\Lambda_{g})\hspace{0.1cm},$$ 
    
    \medskip\noindent then for all $j\geq 1$, $\partial_{\nu}^{j}c=0$ on $\partial M$ and the two operators coincide modulo a smoothing operator : $$\Lambda_{cg}\equiv \Lambda_{g}\hspace{0.1cm}.$$ 
    
    \medskip\noindent In particular, if $c$ is real-analytic and $M$ is connected, then $c\equiv 1$ on $M$.
\end{theoreme}
\noindent The proof relies on a recursive procedure using an adapted version of the aforementioned Duistermaat-Guillemin trace formula and the injectivity of the X-ray transform. Initialization, which requires the use of Liv\v sic theory, is therefore dealt with separately in Section \ref{Section III} and established more precisely in \textbf{Theorem \ref{Th III.18}}.
\\As already mentioned, note that it is possible to remove this condition of equality of metrics at the boundary by adding some relatively strong geometric assumptions, due to the lack of usable spectral invariants. This point is covered in Appendix \ref{App A}.

\medskip\noindent For the linear problem, i.e., considering isospectral deformations, the major difference is that it is possible to establish that for all $s\in\left[ 0,\varepsilon\right]$, 

\begin{equation}\left\{\begin{aligned}
 &\restriction{c_{s}}{\partial M} \equiv 1 \\
&\restriction {\partial_{\nu}c_{s}}{\partial M} \equiv 0 \\
\end{aligned}\right.\label{eq}\end{equation}

\medskip\noindent by differentiating with respect to the parameter $s$ (see \textbf{Proposition \ref{Prop III.1}} and \textbf{Proposition \ref{Th III.13}}).
Then, from there, everything else works the same way at higher order just as in the non-linear problem, so we actually get :

\begin{theoreme}
Let $(M,g)$ be a smooth compact Riemannian manifold of dimension $n\geq 3$ with boundary. Assume that $(\partial M,g^{\circ})$ is Anosov with simple length spectrum. Let $(g_{s})_{0\leqslant s\leqslant \varepsilon}$ be a Steklov isospectral deformation of the metric $g$ s.t there exists a family of smooth functions $\left( c_{s}:M\rightarrow \mathbb{R^{+}_{*}}\right)_{0\leqslant s\leqslant \varepsilon} $ satisfying for any $s\in\left[ 0,\varepsilon\right]$, 
$$g_{s}=c_{s}g\hspace{0.1cm}.$$
If the deformation map $s\mapsto c_{s}$ is real-analytic, then, 
\begin{equation}
    \forall s\in\left[ 0,\varepsilon\right]\hspace{0.1cm},\hspace{0.1cm}\restriction{c_{s}}{\partial M} \equiv 1 \hspace{0.1cm}.\label{eqdeform}
\end{equation}

\medskip\noindent In addition, if \eqref{eqdeform} is verified, then for all $s\in\left[ 0,\varepsilon\right]$, for all $j\geq 1$, $\partial_{\nu}^{j}c_{s}=0$ on $\partial M$ and the two operators coincide modulo a smoothing operator :
$$\Lambda_{g_{s}}\equiv \Lambda_{g}\hspace{0.1cm}.$$ 
In particular, if the conformal factors $c_{s}$ are real-analytic and $M$ is connected, then $c_{s}\equiv 1$ on $M$, for all $s\in\left[ 0,\varepsilon\right]$.
\end{theoreme}  

For the sake of completeness, let us also mention that a simple adaptation of the method established in Section \ref{Section IV} allows us to obtain an analogue of \textbf{Theorem \ref{Th I.1}} in the case of the Steklov inverse problem with potential, i.e. regarding the DN map $\Lambda_{g,q}$ : 

\begin{theoreme}\label{Th I.6}
    Let $(M,g)$ be a smooth compact Riemannian manifold of dimension $n\geq 3$ with boundary. Assume that $(\partial M,g^{\circ})$ is Anosov with simple length spectrum.
    \\Let $q_{1}$ and $q_{2}$ be two smooth potentials on $M$. If the Steklov spectra agree $$\mathrm{spec}(\Lambda_{g,q_{1}})=\mathrm{spec}(\Lambda_{g,q_{2}})\hspace{0.1cm},$$ 

\medskip\noindent then for all $j\geq 1$, $\partial_{\nu}^{j}q_{1}=\partial_{\nu}^{j}q_{2}$ on $\partial M$ and the two operators coincide modulo a smoothing operator : $$\Lambda_{g,q_{1}}\equiv \Lambda_{g,q_{2}}\hspace{0.1cm}.$$ 

\medskip\noindent In particular, if $q_{1}$ and $q_{2}$ are both real-analytic and $M$ is connected, then $q_{1}=q_{2}$ on $M$.    
\end{theoreme}
\medskip\noindent This spectral inverse problem with potential is discussed in Section \ref{Section V}.
\vspace{0.2cm}
\\To conclude, as mentioned in the previous section, recall that the geometric assumption of simple length spectrum required to exploit the wave trace invariants (\textbf{Theorem \ref{Th II.3}} and \textbf{Theorem \ref{Th IV.2}}) is generic in the space of Anosov metrics which makes all the above results, and in particular the non-linear ones, more consistent.
In addition, note that all these results are, to the author's knowledge, the first positive results for the Steklov spectral inverse problem in a general framework.

\begin{Rem}
The following points are worth mentioning here : 
\begin{itemize}
    \item In all these results there is no need to assume boundary connectedness, since we rely on wave trace formula techniques on closed manifolds, and these techniques do not require connectedness of the underlying manifold.
    \item The Anosov assumption is made only on the boundary ; the entire manifold does not need to be Anosov, it can be anything.
    \item The class of manifolds considered here seems rather special, although natural for the method, so it is important to note that one can easily construct such manifolds and there are plenty of them. As a simple example, take any closed Anosov manifold $(M,g)$ of dimension $n\geq 2$ and consider a cylinder $$N:=\left(M,g\right)\times\left(\left[-1,1\right],dt^{2}\right)$$ equipped with the product metric. Then, as mentioned above, it is well-known that one can generically perturb the metric $g$ so that it satisfies the simple length spectrum assumption while preserving the Anosov property.
\end{itemize}
\end{Rem}

\subsection{Plan of the paper}

Section \ref{Section II} introduces the fundamental tools used throughout the paper, mainly the injectivity of the X-ray transform on functions in \ref{Subsection II.1} and spectral invariants for the Steklov spectrum in \ref{Subsection II.2}. Section \ref{Section III} deals with the recovery of the normal derivative of the conformal factor at the boundary, exploiting principal wave trace invariants combined with Liv\v sic theory. Concerning the linear problem, $\left(\ref{eq}\right)$ is completely established by \textbf{Proposition \ref{Th III.13}} and \textbf{Proposition \ref{Prop III.1}}. In Section \ref{Section IV}, we present the recursive procedure that allows us to show more generally that one can recover, in boundary normal coordinates, all the higher-order normal derivatives of the conformal factor at the boundary. This completes the proof of \textbf{Theorem \ref{Th I.1}}. To conclude, in Section \ref{Section V}, we discuss other inverse problems on the Steklov spectrum where the techniques used in this paper could be applied and show in particular \textbf{Theorem \ref{Th I.6}} on the spectral problem for the DN map $\Lambda_{g,q}$. Appendix \ref{App A} gives some statements that establish equality at the boundary of two isospectral conformal metrics, the only point that cannot be dealt with by trace formula techniques. Finally, Appendix \ref{App B} provides an appropriate formula for calculating the subprincipal symbol in \textbf{Lemma \ref{Prop III.12}}.

\subsection*{Acknowledgements}
First and foremost, I would like to thank David dos Santos Ferreira for the many discussions that helped me enormously in writing this paper. I am also grateful to Samuel Tapie for enlightening me on certain notions of hyperbolic dynamics and to Mikko Salo for his pertinent comments and advice on the organization of the document. Thanks to Carolyn S. Gordon for her insight on adapting Sunada's method to produce conformal Steklov isospectral metrics in dimension greater than 2. Finally, I would also like to thank Julien Moy for pointing out the reference \cite{PatPat}, as well as the referee for his valuable comments and suggestions, which helped improve this article.

\subsection*{Notations}
Throughout this article, we denote by $(N,h)$ a smooth closed Riemannian manifold and $Vol_{h}(N)$ its Riemannian volume with respect to the metric $h$. In applications, we will systematically take $(N,h)=(\partial M,g^{\circ})$.
\\We will adopt the classical notations $$Xu:=\mathcal{L}_{X}u=\restriction{\frac{d}{dt}}{t=0}\phi_{t}^{*}u$$ for the Lie derivative of the function $u$ along the vector field $X$, where of course $\phi_{t}^{*}u$ denotes the pullback of $u$ by the flow $\phi_{t}$ generated by $X$.
\\Finally, we will denote by $\Psi^{k}(N)$ the set of pseudodifferential operators of order $k$ on the closed manifold $N$.

\section{Preliminaries}\label{Section II}  

Recall that a flow $\left(\phi_{t}\right)$ generated by a smooth vector field $X$ on a closed manifold $\mathcal{M}$ is Anosov if there is a continuous splitting 
$$T\mathcal{M}=\mathbb{R}X\oplus E_{s}\oplus E_{u}\hspace{0.2cm},$$
of the tangent bundle $T\mathcal{M}$ into three flow-invariant subbundles, respectively, tangential to the flow direction, exponentially contracting and expanding. Note that two Anosov flows that are $C^{1}$-close enough are topologically conjugate : This is known as the \textit{structural stability}. In particular, the Anosov property is robust to $C^{1}$-perturbations, it is a $C^{1}$-open condition with respect to the vector field. We refer, for example, to \cite[Section 5.1]{FisherHass} and to the fundamental \cite{dL} for more details.

A particularly important special case, which will be considered exclusively in this paper, is that where $\mathcal{M}$ is the unit tangent bundle (resp. the unit cotangent bundle)
$$SN:=\lbrace (x,v)\in TN\hspace{0.1cm};\hspace{0.1cm}h_{x}(v,v)=1\rbrace$$

$$\text{(resp. $S^{*}N:=\lbrace (x,\xi)\in T^{*}N\hspace{0.1cm};\hspace{0.1cm}h^{-1}_{x}(\xi,\xi)=1\rbrace$})$$

\medskip\noindent of the closed Riemannian manifold $(N,h)$ and $\left(\phi_{t}\right)$ the associated geodesic flow. We say that $(N,h)$, or simply $h$, is Anosov if its geodesic flow is Anosov. \\As already mentioned, any closed Riemannian manifold with negative curvature is Anosov (see \cite[Theorem 5.2.4]{FisherHass}, \cite[Section 17.6]{KH} or \cite{Kni} for a proof).
\\If the manifold is connected, such flows are transitive\footnote{They admit a dense orbit.}, their periodic orbits are dense, and knowledge of a function along these periodic integral curves allows us to reconstruct it up to natural obstructions. This is the purpose of the standard Liv\v sic theory, and in particular of \textbf{Theorem \ref{Th II.1}} below. More recent elements of the so-called non-Abelian Liv\v sic theory are used via \textbf{Proposition \ref{Prop III.14}}. Finally, as mentioned in the Introduction, the Anosov property provides a general framework particularly suited to obtaining spectral rigidity results by using trace formulas (see Section \ref{Subsection II.2}).

\subsection{X-ray transform and Liv\v sic theorem}\label{Subsection II.1}

We begin by recalling a famous theorem for Anosov flows originally due to Liv\v sic \cite[Theorem 1]{Liv} in H\"older regularity $C^{\alpha}$, $\alpha\in\left(0,1\right)$. We also refer to \cite[Appendix]{GK} and to \cite[Theorem 19.2.4]{KH}. Here we need the smooth regularity version proved in \cite[Theorem 2.1]{dL} :

\begin{theoreme}[Liv\v sic \cite{Liv}, De La Llave-Marco-Moriyon \cite{dL}]\label{Th II.1}
Let $\mathcal{M}$ be a closed Riemannian manifold, $f: \mathcal{M}\rightarrow \mathbb{R}$ a smooth function and $\left(\phi_{s}\right)$ a smooth transitive Anosov flow on $\mathcal{M}$ generated by the vector field $X$.
\vspace{0.2cm}
\\The following propositions are equivalent :
\vspace{0.2cm}
\begin{enumerate}
\item For every periodic orbit $\gamma$ of the flow, $$\int_{\gamma}f = 0\hspace{0.1cm}.$$

\item There is $u\in C^{\infty}(\mathcal{M},\mathbb{R})$ such that $f=Xu$.
\end{enumerate}  
\end{theoreme}

\medskip\noindent This important result characterizes the kernel of a linear integration operator along periodic orbits, called \textit{X-ray transform}. In the following, we recall the definition of this operator for geodesic flows.

\medskip\noindent Let us denote by $\mathcal{C}$ the set of free homotopy classes of $N$. The latter is in $1$-$1$ correspondence with the set of conjugacy classes of the fundamental group $\pi_{1}(N)$ of $N$ and consequently $\mathcal{C}$ is countable. In addition, if $N$ is Anosov, it is well-known that there exists a unique closed geodesic for $h$ in each free homotopy class (see \cite{Kl}, \cite{Kni} or, more generally, \cite[Theorem 29]{CIP}).
\\Thus in this case, the set $\mathcal{G}$ of periodic orbits of the geodesic flow $(\phi_{t})$ can be identified with $\mathcal{C}$.
The \textit{X-ray transform} on $SN$ is the operator defined by :

$$\begin{array}{l|rcl}
\mathcal{I}^{h}: & C^{0}(SN) & \longrightarrow & \ell^{\infty}\left(\mathcal{C}\right) \\
			& f & \longmapsto & \left(\mathcal{I}^{h}f(\left[\gamma\right])\right)_{\left[\gamma\right]\in\mathcal{C}} \end{array}$$
and 
   
   $$\mathcal{I}^{h}f(\left[\gamma\right]):=\frac{1}{\ell_{h}\left(\gamma_{h}\right)}\int_{0}^{\ell_{h}\left(\gamma_{h}\right)}f\left(\phi_{s}\left(x,v\right)\right)ds\hspace{0.1cm},$$

\medskip\noindent\\where $\gamma_{h}:[0, \ell_{h}\left(\gamma_{h}\right)]\rightarrow N$ denotes any parametrization by arclength of the unique closed geodesic in the free homotopy class $\left[\gamma\right]\in\mathcal{C}$, $\ell_{h}$ the length functional for $h$ and $(x,v)\in SN$ is any point on the closed curve $t\mapsto \left(\gamma_{h}(t),\Dot{\gamma}_{h}(t)\right).$

\begin{Rem}\label{Rq II.2}
The map 

$$\begin{array}{l|rcl}
\mathcal{L}_{h}: & \mathcal{C} & \longrightarrow & \left( 0 , \infty\right) \\
			& \left[\gamma\right] & \longmapsto & \ell_{h}\left(\gamma_{h}\right) \end{array}$$

\medskip\noindent is the well-known marked length spectrum of $(N,h)$.
\end{Rem}

\medskip\noindent We usually consider this operator acting on symmetric tensor fields on $N$ : 
\\Let $S^{m}T^{*}N$ be the bundle of symmetric tensors of order $m\in\mathbb{N}$ on $N$, we refer the reader to the recent books \cite{Lef}, or \cite{Guill-Mazz} for definitions in this formalism.
\\Denoting by $C^{\infty}(N,S^{m}T^{*}N)$ the set of smooth sections of $S^{m}T^{*}N\rightarrow N$, we can see any symmetric tensor field $f$ of order $m$ on $N$ as a function on $SN$ by considering the lift $\pi^{*}_{m}:C^{\infty}(N,S^{m}T^{*}N)\rightarrow C^{\infty}(SN)$ defined by 
$$\pi_{m}^{*}f:(x,v)\mapsto f_{x}\left(v,...,v\right)\hspace{0.1cm},$$ 

\medskip\noindent with the convention that $S^{0}T^{*}N:=N\times\mathbb{R}$ is the trivial bundle, so that $\pi^{*}_{0}$ is simply the pullback by the projection $\pi:SN\rightarrow N$.
\medskip\noindent\\Then with the same notations as above, the \textit{X-ray transform} on symmetric tensors of order $m$ is the operator 

$$\begin{array}{l|rcl}
\mathcal{I}_{m}^{h}:=\mathcal{I}^{h}\circ\pi_{m}^{*} & C^{\infty}(N,S^{m}T^{*}N) & \longrightarrow & \ell^{\infty}\left(\mathcal{C}\right) \\
			& f & \longmapsto & \left(\mathcal{I}^{h}\pi_{m}^{*}f(\left[\gamma\right])\right)_{\left[\gamma\right]\in\mathcal{C}} \end{array}\hspace{0.1cm},$$
   where $$\mathcal{I}^{h}\pi_{m}^{*}f(\left[\gamma\right]):=\frac{1}{\ell_{h}\left(\gamma_{h}\right)}\int_{0}^{\ell_{h}\left(\gamma_{h}\right)}f_{\gamma_{h}(t)}\left(\Dot{\gamma}_{h}(t),...,\Dot{\gamma}_{h}(t)\right)dt\hspace{0.1cm}.$$
\\As mentioned in the Introduction, the problem considered here will mainly involve $\mathcal{I}_{0}^{h}$, the X-ray transform on functions which is injective in the case of Anosov manifolds :

\begin{theoreme}[\cite{DairbSha}]\label{Th II.5}
Suppose $h$ is Anosov. Then for any $f\in C^{\infty}(N)$, 

$$\mathcal{I}_{0}^{h}f=0\hspace{0.2cm}\Rightarrow\hspace{0.2cm}f=0\hspace{0.1cm}.$$   
\end{theoreme}

The conclusion of \textbf{Theorem \ref{Th II.5}} is a simple special case of a more general notion of injectivity on symmetric tensors of order $m$, called s-injectivity. The literature on this subject is quite extensive; we refer, for example, to \cite{DairbSha}, \cite{CrSh} and the more recent \cite{PatSalUhl2} and \cite{Guillarmou}.  Its study in the context of general Anosov manifolds is still largely open, at least in dimension $n\geqslant 3$, even if major progress has recently been made by Ceki{\'c} and Lefeuvre in \cite{CekiLef}.

\subsection{Spectral invariants for the DN map}\label{Subsection II.2}

As in the Introduction, let us consider $\left(M,g\right)$ a smooth compact Riemannian manifold of dimension $n\geqslant 2$ with boundary $(N,h):=(\partial M,g^{\circ})$.
The aim of this section is to introduce the spectral invariants used throughout this paper. The first and probably the best known, which will be used in Appendix \ref{App A}, follows from the asymptotic of the spectral counting function of the DN map ; namely applying standard results of \cite[Chapter 29]{Horm2} we have the following distribution of Steklov eigenvalues :

\begin{Prop}[Weyl's law]\label{Prop II.1}
$$\#(\sigma_{k}\leq \sigma)=\frac{\mathrm{Vol}(\mathbb{B}^{n-1})}{(2\pi)^{n-1}}\mathrm{Vol}_{g^{\circ}}(\partial M)\sigma^{n-1}+\mathcal{O}(\sigma^{n-2})\hspace{0.1cm},$$ 
where $\mathbb{B}^{n-1}$ denotes the unit ball in $\mathbb{R}^{n-1}$.
\end{Prop}

\medskip\noindent Consequently, since two isospectral manifolds have the same spectral counting function :

\begin{coro}\label{Coro II.2}
Given any smooth compact Riemannian manifold of dimension $n\geqslant 2$ with boundary, the volume of its boundary is a spectral invariant for the Steklov spectrum.    
\end{coro}

We now introduce the invariants on which this paper is based,  provided by the fundamental wave trace formula of Duistermaat-Guillemin \cite[Theorem 4.5]{GD}. The latter, applied to the DN map, is related to the trace of the wave group 
$$U(t):=e^{-it\Lambda_{g}}\hspace{0.1cm}.$$
It shows in particular that the singularities of 
$$\mathrm{Tr}\left(U(t)\right):=\sum_{k=0}^{+\infty}e^{-i\sigma_{k}t}$$
as a distribution in time, lie at the periods of the periodic geodesics of the boundary, in other words :
$$Suppsing\hspace{0.05cm}\mathrm{Tr}\left(U(t)\right)\subseteq \mathcal{T}\hspace{0.1cm},$$

\medskip\noindent where $Suppsing$ denotes the singular support of a distribution (the complement of the set where it is smooth) and $\mathcal{T}$ is the set of periods. Any element of the set $\lvert\mathcal{T}\rvert:=\lbrace \lvert T\rvert\hspace{0.1cm};\hspace{0.1cm}T\in\mathcal{T}\rbrace$ will, of course, be referred to as the length of a periodic orbit and $\lvert\mathcal{T}\rvert$ will be called the \emph{length spectrum}.
\\In order to state this wave trace formula more precisely, we need to recall a few concepts : 

\begin{itemize}
\item Given a closed orbit $\gamma$ of period $T_{\gamma}$, its primitive period $T_{\gamma}^{\sharp}$ is the least non-zero period such that $T_{\gamma}$ is a multiple of $T_{\gamma}^{\sharp}$.

    \item The linearized Poincaré return map $\mathcal{P}_{\gamma}$ of the closed orbit $\gamma$ of period $T_{\gamma}$ starting from $w=(x,\xi)\in S^{*}N$, is the linear application induced by $d\phi_{T_{\gamma}}(w)$ in $T_{w}S^{*}N/\mathbb{R}X(w)$ where $X$ denotes, of course, the generator of the geodesic flow of $N$. A closed geodesic $\gamma$ is called \textit{non-degenerate} if 
$$\det\left(Id-\mathcal{P}_{\gamma}\right)\neq 0\hspace{0.1cm}.$$
In particular, it is easy to see that this property holds if the underlying manifold $N$ is Anosov, since in this case $\mathcal{P}_{\gamma}$ may be identified with $\restriction{d\phi_{T_{\gamma}}(w)}{E_{s}(w)\oplus E_{u}(w)}\hspace{0.1cm}.$
    \item The Morse index $m_{\gamma}$ of a closed geodesic $\gamma$ is the maximal dimension of a vector subspace on which the second-order differential of the energy functional at $\gamma$ is negative definite. We refer to \cite{Duis} for details.
\end{itemize}
For further information on these concepts, we refer more generally to the online book \cite{Guill-Mazz} currently in progress.

\begin{Theor}[Duistermaat-Guillemin, \cite{GD}]\label{Th II.3}

Suppose that $(\partial M,g^{\circ})$ only admits a finite number of closed geodesics $\gamma_{1},...,\gamma_{r}$ of period $T$ and that each of these geodesics is nondegenerate.
\\Then
   $$S(t):=\mathrm{Tr}\left(U(t)\right)$$ 
   \\is a tempered distribution on $\mathbb{R}$ admitting an isolated singularity in $t=T$ given by a Lagrangian distribution $e_{T}$ with 
   $$Suppsing\left(e_{T}\right)\subseteq\lbrace T\rbrace$$  
and complete asymptotic expansion $$e_{T}(t)\underset{t\rightarrow T}{\sim} \hspace{0.1cm} c_{T,-1}(t-T+i0)^{-1}+\sum_{j=0}^{+\infty}c_{T,j}(t-T+i0)^{j}\log(t-T+i0)\hspace{0.1cm}.$$ 
The leading coefficient is expressed as 
   
   $$c_{T,-1}=\sum_{j=1}^{r}\frac{\lvert T_{j}^{\sharp}\rvert e^{im_{j}\frac{\pi}{2}}e^{-iT\hspace{0.05cm}\overline{\gamma_{j}}}}{\lvert \det\left(Id-\mathcal{P}_{j}\right)\rvert^{1/2}}\hspace{0.1cm},$$ 
   
   \medskip\noindent where 
   $$\overline{\gamma_{j}}:=\frac{1}{T}\int_{\gamma_{j}}sub(\Lambda_{g})\hspace{0.1cm},$$ 
   and $\mathcal{P}_{j}$, $T_{j}^{\sharp}$, $m_{j}$ are respectively the linearized Poincaré map, the primitive period and the Morse index of $\gamma_{j}$.
\end{Theor}

\begin{Rem}\label{Rq II.9}
Note that the Morse index is zero for any non-degenerate closed orbit with no conjugate point, which is the case in particular for all closed geodesics of an Anosov manifold. We refer again to \cite{Duis} for details and, for a proof of this result, the reader can consult the aforementioned \cite{Guill-Mazz}. We also refer to \cite{Kli} and more generally to \cite{PatPat}.
\end{Rem}

\medskip\noindent The leading term $c_{T,-1}$ in the previous theorem is called the principal wave trace invariant at $t=T$, and is the term controlling the singularity. In particular, if this term is cancelled, then the corresponding period $T$ is not a singularity. \\Consequently, to take full advantage of this theorem, we systematically consider the general case where $(\partial M,g^{\circ})$ is Anosov with the generic assumption that the length spectrum is simple (i.e. $r=1$ in the theorem).
\\It is then possible to formulate a result analogous to \cite[Theorem 3]{GK} in our context, highlighting two spectral invariants :

\begin{coro}\label{Coro II.8}
In the class of Riemannian metrics $g$ such that $g^{\circ}$ is Anosov with simple length spectrum, the Steklov spectrum determines :
\begin{itemize}
    \item[(a)] the lengths of the closed geodesics on the boundary.
    \item[(b)] for any period $T\in\mathcal{T}$, the quantity 

    $$c_{T,-1}=\frac{\lvert T^{\sharp}\rvert exp\left(-i\hspace{0.05cm}\int_{\gamma_{g^{\circ}}}sub(\Lambda_{g})\right)}{\lvert \det\left(Id-\mathcal{P}_{\gamma_{g^{\circ}}}\right)\rvert^{1/2}}\hspace{0.1cm}\cdot$$
\end{itemize}
\end{coro} 

\begin{Rem}
As already mentioned, in the standard case of the Laplacian on a closed manifold, the subprincipal symbol is zero and consequently the Anosov property (or the negative curvature in the case of \cite{GK}) is sufficient to ensure, in this context, the point (a) of the previous corollary. Indeed in this case, given the expression of the principal wave trace invariant, no cancellation takes place in the wave trace formula i.e. every period contributes to a singularity. \\In our case however, the subprincipal symbol is non-zero, which forces us to ensure that there are no such cancellations by adding the simple length spectrum assumption. On the other hand, it also provides additional information, namely point (b).
\end{Rem}

To conclude this section, note that it is possible to show in a similar way to the spectral problem on the Laplacian (see \cite{GK} or \cite{CrSh}), that infinitesimal spectral rigidity is actually related to infinitesimal rigidity of the marked length spectrum, this is the object of the following.

\medskip\noindent If $(g_{s})_{0\leqslant s\leqslant\varepsilon}$ is any deformation of $g$ (see \textbf{Definition \ref{Def I.1}}) s.t $g_{s}^{\circ}$ is an Anosov metric for all $s$, any closed geodesic $\gamma$ for $g^{\circ}$ deforms smoothly as a closed geodesic $\gamma_{s}$ for $g_{s}^{\circ}$ and its length is a smooth function in $s$ (cf. \cite{dL}, or \cite[Lemma 15.3]{Guill-Mazz} for a precise statement in this specific context). If, in addition, $(g_{s})_{0\leqslant s\leqslant\varepsilon}$ is Steklov isospectral and $g^{\circ}$ has simple length spectrum, one can show that the marked length of a closed geodesic $\gamma$ is locally preserved along the deformation :

\begin{lem}\label{Lemma II.10}
For any free homotopy class $\left[\gamma\right]\in\mathcal{C}$, there exists $\eta_{\left[\gamma\right]}\in\left]0,\varepsilon\right]$ such that
$$\forall s\leq\eta_{\left[\gamma\right]}\hspace{0.1cm},\hspace{0.1cm}\mathcal{L}_{g_{s}^{\circ}}(\left[\gamma\right])=\mathcal{L}_{g^{\circ}}(\left[\gamma\right])\hspace{0.1cm}.$$    
\end{lem}

\medskip\noindent The key point in the proof of this result, is to see that the length of any closed geodesic remains uniformly isolated under small perturbations of the metric. This is a consequence of the work of Guillarmou and Lefeuvre \cite[Proposition 2.1]{GuiLef} :

\begin{lem}\label{Lemma II.11}
For any free homotopy class $\left[\gamma\right]\in\mathcal{C}$, there exists $s_{0}\in\left]0,\varepsilon\right[$ and $\delta>0$ such that for all $s<s_{0}$, 
$$I_{s}\cap\lvert\mathcal{T}_{g_{s}^{\circ}}\rvert=\lbrace\mathcal{L}_{g_{s}^{\circ}}(\left[\gamma\right])\rbrace\hspace{0.1cm},$$ 
where $$I_{s}:=\left]\mathcal{L}_{g_{s}^{\circ}}(\left[\gamma\right])-\delta\hspace{0.1cm},\hspace{0.1cm}\mathcal{L}_{g_{s}^{\circ}}(\left[\gamma\right])+\delta\right[$$ 

\medskip\noindent and $\lvert\mathcal{T}_{g_{s}^{\circ}}\rvert$ denotes the set of lengths of closed geodesics for $g_{s}^{\circ}$.    
\end{lem}
\begin{proof}
Let $\left[\gamma\right]\in\mathcal{C}$. Since $g^{\circ}$ is Anosov, its length spectrum is discrete and consequently, there exists $\eta>0$ s.t \begin{equation}
\left]\mathcal{L}_{g^{\circ}}(\left[\gamma\right])-\eta\hspace{0.1cm},\hspace{0.1cm}\mathcal{L}_{g^{\circ}}(\left[\gamma\right])+\eta\right[\cap\lvert\mathcal{T}_{g^{\circ}}\rvert=\lbrace\mathcal{L}_{g^{\circ}}(\left[\gamma\right])\rbrace\hspace{0.1cm}.\label{eqisol}    
\end{equation}
By \cite[Proposition 2.1]{GuiLef}, the $g^{\circ}$-\emph{normalized marked length spectrum} converges in $\ell^{\infty}\left(\mathcal{C}\right)$ as $s$ tends to $0$
$$\left\lVert\left(\frac{\mathcal{L}_{g_{s}^{\circ}}(\left[\omega\right])}{\mathcal{L}_{g^{\circ}}(\left[\omega\right])}\right)_{\left[\omega\right]\in\mathcal{C}}-\mathbf{1}\right\rVert_{\ell^{\infty}\left(\mathcal{C}\right)}\underset{s\rightarrow 0}{\longrightarrow}\quad 0\hspace{0.1cm}.$$
Consequently, for all $\epsilon>0$ and $s$ small enough, \begin{equation}
(1-\epsilon)\mathcal{L}_{g^{\circ}}\leq\mathcal{L}_{g_{s}^{\circ}}\leq (1+\epsilon)\mathcal{L}_{g^{\circ}}\hspace{0.1cm}.\label{eqnorm}  \end{equation}
Taking $\epsilon>0$ sufficiently small, shrinking $\eta$ if necessary, and using the continuity of $s\mapsto\mathcal{L}_{g_{s}^{\circ}}(\left[\gamma\right])$ at $0$, we can ensure that there exists $s_{0}>0$ such that for any $s<s_{0}$, $$\left]\frac{\mathcal{L}_{g_{s}^{\circ}}(\left[\gamma\right])-\delta}{1+\epsilon}\hspace{0.1cm},\hspace{0.1cm}\frac{\mathcal{L}_{g_{s}^{\circ}}(\left[\gamma\right])+\delta}{1-\epsilon}\right[\subset\left]\mathcal{L}_{g^{\circ}}(\left[\gamma\right])-\eta\hspace{0.1cm},\hspace{0.1cm}\mathcal{L}_{g^{\circ}}(\left[\gamma\right])+\eta\right[\hspace{0.1cm},$$ 
with $\delta:=\eta/2$.
Assuming then that there exists $\left[\omega\right]\neq\left[\gamma\right]$ such that $$\mathcal{L}_{g_{s}^{\circ}}(\left[\omega\right])\in\left]\mathcal{L}_{g_{s}^{\circ}}(\left[\gamma\right])-\delta\hspace{0.1cm},\hspace{0.1cm}\mathcal{L}_{g_{s}^{\circ}}(\left[\gamma\right])+\delta\right[\hspace{0.1cm},$$ \eqref{eqnorm} would lead to $$\mathcal{L}_{g^{\circ}}(\left[\omega\right])\in\left]\mathcal{L}_{g_{s}^{\circ}}(\left[\gamma\right])-\eta\hspace{0.1cm},\hspace{0.1cm}\mathcal{L}_{g_{s}^{\circ}}(\left[\gamma\right])+\eta\right[\hspace{0.1cm},$$ 

\medskip\noindent which would contradict \eqref{eqisol}. 

\end{proof}

\medskip\noindent Let us point out that the proof of \textbf{Lemma \ref{Lemma II.11}} does not use the simple length spectrum assumption on $g^{\circ}$ at all.

\begin{proof}[Proof of \textbf{Lemma \ref{Lemma II.10}}]
Let $\left[\gamma\right]\in\mathcal{C}$ be any free homotopy class. Using \textbf{Lemma \ref{Lemma II.11}} and its notations, there exists $\delta>0$ s.t for any $s\in \left[0,\varepsilon\right]$,
$$I_{s}\cap\lvert\mathcal{T}_{g_{s}^{\circ}}\rvert=\lbrace\mathcal{L}_{g_{s}^{\circ}}(\left[\gamma\right])\rbrace\hspace{0.1cm}.$$ 
In addition, denoting also by $\lvert\mathcal{T}_{g^{\circ}}\rvert$ the length spectrum of $g^{\circ}$, we have \begin{equation}
\lvert\mathcal{T}_{g^{\circ}}\rvert\subseteq\lvert\mathcal{T}_{g_{s}^{\circ}}\rvert\hspace{0.1cm}.\label{eqSpec}    
\end{equation}
This fact is a direct consequence of isospectrality combined with the wave trace formula for the DN map (cf. \textbf{Theorem \ref{Th II.3}} above) which gives the inclusion of the singular supports of the wave traces associated with $\Lambda_{g_{s}}$ and $\Lambda_{g}$. Note that \eqref{eqSpec} is a priori just an inclusion because the length spectrum of $g_{s}^{\circ}$ is not necessarily simple.

\medskip\noindent To conclude, we use that the map $s\mapsto\mathcal{L}_{g_{s}^{\circ}}(\left[\gamma\right])$ is continuous (even smooth as mentioned above). In particular, there is 
$\eta_{\left[\gamma\right]}\in\left]0,\varepsilon\right]$ s.t for any $s\leq\eta_{\left[\gamma\right]}\hspace{0.1cm},$ 

$$\lvert\mathcal{L}_{g_{s}^{\circ}}(\left[\gamma\right])-\mathcal{L}_{g^{\circ}}(\left[\gamma\right])\rvert\leq\delta$$

\medskip\noindent and since $\mathcal{L}_{g^{\circ}}(\left[\gamma\right])\in\lvert\mathcal{T}_{g_{s}^{\circ}}\rvert$ (by \eqref{eqSpec}), we get that for any $s\leq\eta_{\left[\gamma\right]}\hspace{0.1cm},$ 

$$\mathcal{L}_{g_{s}^{\circ}}(\left[\gamma\right])=\mathcal{L}_{g^{\circ}}(\left[\gamma\right])\hspace{0.1cm}.$$
\end{proof}
\medskip\noindent Note that this result is less direct and much weaker than in the case of Laplace isospectral deformations, since in the latter, inclusion (\ref{eqSpec}) is an equality for any $s$ and the length spectrum does not need to be simple. \\In relation to this last point, we would like to point out to the reader that \textbf{Lemma \ref{Lemma II.10}} does not imply that the simple length spectrum assumption is preserved along the deformation even for $s$ sufficiently small. For this to be the case, $\eta$ would have to be uniform (independent of any free homotopy class $\left[\gamma\right]$) which is absolutely not true in general even in the class of Steklov isospectral metrics. This is a significant obstacle to linearizing the problem, but we can nevertheless consider a weaker version, which is the subject of the following discussion. 

\medskip\noindent Differentiating for each free homotopy class the equality of marked length spectra (given by \textbf{Lemma \ref{Lemma II.10}}) with respect to the parameter, we obtain that the first variation 
$$\dot{g}^{\circ}=\restriction{\frac{\partial}{\partial s}}{s=0}g_{s}^{\circ}$$
has to be in the kernel of the X-ray transform $\mathcal{I}_{2}^{g^{\circ}}$ :

\begin{Prop}\label{Prop II.7}
For any free homotopy class $\left[\gamma\right]\in\mathcal{C}$,
$$\mathcal{I}_{2}^{g^{\circ}}\dot{g}^{\circ}(\left[\gamma\right])=0\hspace{0.1cm}.$$     
\end{Prop}  

\medskip\noindent A clean proof of this interesting relation between the kernel of the X-ray transform and the marked length spectrum can be found in \cite[Lemma 15.4]{Guill-Mazz}. It will serve as the initialization step for establishing \textbf{Proposition \ref{Prop III.1}} by induction on the order of derivation of the metric tensor. Indeed, denoting for all $k\in\mathbb{N}_{\geq 2}$, $$\left(g^{\circ}\right)^{(k)}=\restriction{\frac{\partial^{k}}{\partial s^{k}}}{s=0}g_{s}^{\circ}$$

\medskip\noindent the $k$-th variation of the metric tensor and differentiating $k$ times the equality of \textbf{Lemma \ref{Lemma II.10}} with respect to $s$, we observe more generally that for all $\left[\gamma\right]\in\mathcal{C}$, 
\begin{equation}\mathcal{I}_{2}^{g^{\circ}}\left(g^{\circ}\right)^{(k)}(\left[\gamma\right])\hspace{0.1cm}+\hspace{0.1cm}\mathcal{R}_{k}(\left[\gamma\right])\left[\dot{g}^{\circ}, \left(g^{\circ}\right)^{(2)},...,\left(g^{\circ}\right)^{(k-1)}\right]=0\hspace{0.1cm},\label{eqmarked}\end{equation}
where the second term is a finite sum of integrals over the closed $g^{\circ}$-geodesic $\gamma$ of (tensorial) polynomials in the lower-order variations $\dot{g}^{\circ}$, $\left(g^{\circ}\right)^{(2)}$,$...$, $\left(g^{\circ}\right)^{(k-1)}$ and their covariant derivatives of order at most $k-2$ along $\gamma$. The left-hand side of \eqref{eqmarked} follows from a technical but standard calculation of the $k$-th derivative of $s\mapsto\mathcal{L}_{g_{s}^{\circ}}(\left[\gamma\right])$ with respect to $s$, which, on the other hand, is zero at $s=0$ according to \textbf{Lemma \ref{Lemma II.10}}.

\medskip\noindent In particular, considering a deformation of conformal metrics $(g_{s}):=(c_{s}g)$, any variation $\left(g^{\circ}\right)^{(m)}$ is reduced to $\left(\restriction{\partial_{s}^{m}c_{s}}{s=0}\right)g^{\circ}$,  so if each of these derivatives is zero at the boundary up to order $k$, we have :

\begin{equation}\forall\left[\gamma\right]\in\mathcal{C}\hspace{0.1cm},\hspace{0.1cm}\mathcal{I}_{0}^{g^{\circ}}\restriction{\partial_{s}^{k}c_{s}}{s=0}(\left[\gamma\right])=0\hspace{0.1cm}.\label{eqind}\end{equation}
Finally, it should be noted that in all this, it is not essential to assume that the deformation preserves the Anosov property at the boundary, due to the structural stability of an Anosov flow and the fact that we differentiate at $s=0$. 

\section{Recovery of the normal derivative at the boundary}\label{Section III}

In this section we show that the information provided at subprincipal order concerns the normal derivative of the conformal factor at the boundary, and we explain how it is possible to recover it from the Steklov spectrum. More precisely, we show here that if $g$ and $cg$ are isospectral, coincide at the boundary, and $g^{\circ}$ is Anosov with simple length spectrum, then 
$$\restriction{\partial_{\nu}c}{\partial M}\equiv 0\hspace{0.1cm}.$$
As mentioned in the Introduction, this spectral inverse problem is a priori non-linear and we are actually unable to linearize it using the recursive procedure presented in Section \ref{Section IV}, which allows us to deal with the lower-order levels. \\In this section, we first formulate a result in the context of isospectral deformations (linear problem) as introduced in \textbf{Definition \ref{Def I.1}}, the deformation parameter $s$ allowing the problem to be linearized by differentiation (see \textbf{Proposition \ref{Th III.13}}).
\\We then prove the result for the general problem, bypassing the absence of a differentiation parameter through the non-Abelian Liv\v sic theory (we refer the reader to \textbf{Proposition \ref{Prop III.14}} below).

\medskip\noindent From now on, we therefore assume that $\restriction{c}{\partial M}\equiv 1$ and begin by showing that the exploitable quantity that appears when calculating the subprincipal symbol of the DN map $\Lambda_{cg}\in\Psi^{1}(\partial M)$ is, as mentioned above, the normal derivative of the conformal factor at the boundary :

\begin{Lem}\label{Prop III.12}
Let $c:M\rightarrow \mathbb{R}^{+}_{*}$ be a smooth function satisfying $\restriction{c}{\partial M} \equiv 1$. 

\medskip\noindent Then we have the following equality on $T^{*}\partial M$ 
    
$$sub(\Lambda_{cg})-sub(\Lambda_{g})=\alpha_{n}\restriction{\left(\partial_{\nu}c\right)}{\partial M}\hspace{0.1cm},$$
where $\alpha_{n}=-\dfrac{n-2}{4}\cdot$
\end{Lem}
\begin{proof}
If we note in any local coordinates 
$$\sigma^{full}_{\Lambda_{g}}\sim \sum_{j=-1}^{+\infty}p_{-j}\hspace{0.2cm},$$
in the sense of \eqref{eq1}, then by the result of \textbf{Corollary \ref{coro 1.3.1}} and the expression of the subprincipal symbol in coordinates (see the beginning of Section \ref{Section IV}), we obtain the following identity : $$\begin{aligned}[t]sub(\Lambda_{cg}) &=p_{0}+\alpha_{n}\restriction{\left(\partial_{\nu}c\right)}{\partial M}-\frac{1}{2i}\sum_{j=1}^{n}\frac{\partial^{2}\sigma_{\Lambda_{g}}}{\partial x_{j}\partial\xi_{j}} \\
&= sub(\Lambda_{g})+\alpha_{n}\restriction{\left(\partial_{\nu}c\right)}{\partial M}\\
\end{aligned}$$   

\medskip\noindent in any coordinate chart $(x,\xi)$ on $T^{*}\partial M$.
\end{proof}

\medskip\noindent Combining this with \cite[Theorem $4.5$]{GD} (see \textbf{Theorem \ref{Th II.3}}), we are able to prove the second part of the infinitesimal spectral rigidity result $\left(\ref{eq}\right)$ : 

\begin{prop}\label{Th III.13}
Let $(M,g)$ be a smooth compact Riemannian manifold with boundary such that $(\partial M,g^{\circ})$ is Anosov with simple length spectrum.\\Let $(g_{s})_{0\leqslant s\leqslant \varepsilon}$ be a Steklov isospectral deformation of the metric $g$ s.t there exists a family of smooth functions $\left( c_{s}:M\mapsto \mathbb{R^{+}_{*}}\right)_{0\leqslant s\leqslant \varepsilon} $ satisfying for any $s\in\left[ 0,\varepsilon\right]$, 

$$g_{s}=c_{s}g$$
and $$\restriction{c_{s}}{\partial M} \equiv 1\hspace{0.1cm}.$$
Then for all $s\in\left[ 0,\varepsilon\right]$,
$$\restriction {\partial_{\nu}c_{s}}{\partial M}\equiv 0\hspace{0.1cm}.$$
\end{prop}
\begin{proof}
For all $s\in\left[ 0,\varepsilon\right]$, $g_{s}^{\circ}=g^{\circ}$, so the closed geodesics for $g_{s}^{\circ}$ are exactly the closed geodesics for $g^{\circ}$ and since $g_{s}^{\circ}$ and $g^{\circ}$ are Anosov with simple length spectrum, then by isospectrality, the principal wave trace invariant given by \textbf{Theorem \ref{Th II.3}} remains constant : 
$$\frac{\lvert T_{\gamma}^{\sharp}\rvert e^{-i\int_{\gamma}sub(\Lambda_{g_{s}})}}{\lvert \det\left(Id-\mathcal{P}_{\gamma}\right)\rvert^{1/2}}=\frac{\lvert T_{\gamma}^{\sharp}\rvert e^{-i\int_{\gamma}sub(\Lambda_{g})}}{\lvert \det\left(Id-\mathcal{P}_{\gamma}\right)\rvert^{1/2}}\hspace{0.1cm},$$

\medskip\noindent \\for any $\gamma:=\gamma_{0}$ closed orbit of the (co)geodesic flow for $g^{\circ}$ with period $T_{\gamma}$ (recall that $T_{\gamma}^{\sharp}$ is its primitive period).
\\Thus for all $s\in\left[ 0,\varepsilon\right]$ and any $\gamma$ closed orbit,
$$exp\left(-i\int_{\gamma}\left[sub(\Lambda_{g_{s}})-sub(\Lambda_{g})\right]\right)=1$$ 
and then by \textbf{Lemma \ref{Prop III.12}} we obtain $$exp\left(-i\alpha_{n}\int_{\gamma}\restriction{\left(\partial_{\nu}c_{s}\right)}{\partial M}\right)=1\hspace{0.1cm}.$$
Differentiating with respect to $s$, we finally get for any $\left[\gamma\right]\in\mathcal{C}$,
$$\mathcal{I}_{0}^{g^{\circ}}\partial_{\nu}\dot{c}_{s}\left(\left[\gamma\right]\right)=\int_{\gamma}\restriction{\left(\partial_{\nu}\dot{c}_{s}\right)}{\partial M}=0\hspace{0.1cm}.$$
\medskip\noindent From there we can conclude that for all $s\in \left[0,\varepsilon\right]$, 
$$\restriction{\left(\partial_{\nu}\dot{c}_{s}\right)}{\partial M}=0$$ 

\medskip\noindent in the same way as in \textbf{Proposition \ref{Prop III.1}} by the injectivity of the X-ray transform on functions for Anosov manifolds (\textbf{Theorem \ref{Th II.5}}). Hence for all $s\in \left[0,\varepsilon\right]$, $\partial_{\nu}c_{s}=\partial_{\nu}c_{0}=0$ on $\partial M$.
\end{proof}

\medskip\noindent We now consider the non-linear problem. In this case, the information given by the principal wave invariant on the X-ray transform is more delicate to exploit.
\\First, to make this information explicit, we recall a result analogous to the standard Liv\v sic theorem \textbf{Theorem \ref{Th II.1}} :

\begin{prop}\label{Prop III.14}
Let $\mathcal{M}$ be a closed Riemannian manifold, $f: \mathcal{M}\rightarrow \mathbb{R}$ be a smooth function, and $\left(\phi_{s}\right)$ a smooth transitive Anosov flow on $\mathcal{M}$ generated by the vector field $X$.
\vspace{0.2cm}
\\The following propositions are equivalent :
\vspace{0.2cm}
\begin{enumerate}
\item For every periodic orbit $\gamma$ of the flow, 

$$\int_{\gamma}f\in 2\pi\mathbb{Z}\hspace{0.1cm}.$$
\item There is $u\in C^{\infty}(\mathcal{M},\mathbb{S}^{1})$ such that $$f=-i\hspace{0.05cm}\frac{Xu}{u}\hspace{0.1cm}\cdot$$
\end{enumerate}
\end{prop}

\medskip\noindent This statement follows directly from the use of Liv\v sic theorem for cocycles, we refer the reader to \cite[Theorem 2.2]{Pat} and references therein. Of course, it is possible to prove this result in a more standard way by following the same initial proof as that of Liv\v sic theorem (\textbf{Theorem \ref{Th II.1}}), i.e. by first showing the statement in the H\"older $C^{\alpha}$ regularity and then extending it to smooth regularity. The latter can also be shown using the more recent microlocal approach of \cite{Guillarmou}.
\vspace{0.2cm}
\\Finally, note that our problem is similar to what appears when considering two transparent connections (see \cite[Section 3]{Pat}) on a Hermitian line bundle over an Anosov manifold as in the proof of \cite[Theorem 3.2]{Pat}. The latter shows that two transparent connections are gauge-equivalent. It is based on the use of a standard algebraic tool called Gysin sequences (see \cite[Chapter 14, p.177-179]{BottTu}), reducing the proof to the s-injectivity of $\mathcal{I}_{1}$ on Anosov manifolds (see \cite[Theorem 1.3]{CrSh}). \\The following theorem can be proved using essentially the same type of arguments.

\begin{Theo}\label{Th III.17}
    Let $(N,h)$ be a closed Anosov manifold, $X$ the generator of its geodesic flow and $f$ a smooth real-valued function on $N$.\\If $u\in C^{\infty}(SN)$ is a non-trivial solution to the following transport equation with potential $f$ 
    $$Xu+ifu=0\hspace{0.1cm},$$
then
    $$f\equiv 0\hspace{0.1cm}.$$
\end{Theo}

\medskip\noindent This result is actually related to a much more general theory (tensor tomography for connections) developed over the last
ten years following the work of Guillarmou-Paternain-Salo-Uhlmann. 
\vspace{0.2cm}
\\We can now show the analog of \textbf{Proposition \ref{Th III.13}} in the non-linear case : 

\begin{Theo}\label{Th III.18}
    Let $(M,g)$ be a smooth compact Riemannian manifold of dimension $n\geq 3$ with smooth boundary. Assume that $(\partial M,g^{\circ})$ is Anosov with simple length spectrum. Let $c:M\rightarrow \mathbb{R}^{+}_{*}$ be a smooth function satisfying $\restriction{c}{\partial M} \equiv 1\hspace{0.1cm}.$ 
    \\If 
    $$\mathrm{spec}(\Lambda_{cg})=\mathrm{spec}(\Lambda_{g})\hspace{0.1cm},$$ 
then 
    $$\restriction{\partial_{\nu}c}{\partial M} \equiv 0\hspace{0.1cm}.$$    
\end{Theo}

\begin{proof}
Since $g$ and $cg$ are two isospectral metrics that coincide at the boundary and $g^{\circ}$ is Anosov with simple length spectrum, \textbf{Theorem \ref{Th II.3}} and \textbf{Lemma \ref{Prop III.12}} implies that for any closed geodesic $\gamma$ for $g^{\circ}$, 

$$exp\left(-i\alpha_{n}\int_{\gamma}\restriction{\left(\partial_{\nu}c\right)}{\partial M}\right)=1$$

\medskip\noindent in the same way as in the proof of \textbf{Proposition \ref{Th III.13}}. 

\medskip\noindent Equivalently, if we put $f:=\alpha_{n}\restriction{\left(\partial_{\nu}c\right)}{\partial M}$, then for any closed geodesic $\gamma$,

$$\int_{\gamma}f\in 2\pi\mathbb{Z}\hspace{0.1cm}.$$

\medskip\noindent Combining \textbf{Proposition \ref{Prop III.14}} and \textbf{Theorem \ref{Th III.17}}, we obtain 
$f\equiv 0\hspace{0.1cm}.$
\end{proof}

\section{Boundary determination of the conformal factor}\label{Section IV}

This section establishes the important results of boundary determination, namely that the spectrum determines the jet of the conformal factor at the boundary. \\When $g$ and $cg$ are Steklov isospectral, we showed at the end of Section \ref{Section III}, under the assumption that $\Lambda_{cg}$ and $\Lambda_{g}$ coincide at the principal order, that the subprincipal part of $\Lambda_{cg}-\Lambda_{g}$ is zero, which has given the nullity of the normal derivative of $c$ at the boundary. \\Here we show more generally that the normal derivative at the boundary at any order $j\geq 2$ is determined by the homogeneous term of order $-j+1$ in the symbol of $\Lambda_{cg}-\Lambda_{g}$ and that this term is zero under the assumption of isospectrality of the metrics. This is established recursively on the order, thanks to an adapted version of the Duistermaat-Guillemin trace formula \cite[Theorem 4.5]{GD} generalizing Guillemin's \cite[Theorem 4]{Gui}, in combination with the injectivity of the X-ray transform \textbf{Theorem \ref{Th II.5}}.
\\In the following, we introduce the necessary formalism and refer the reader to the fundamental works \cite{DH} and \cite{Horm2} for details on the theory of Fourier integral operators (FIO).
\vspace{0.2cm}

First, recall that a pseudodifferential operator $P\in\Psi^{m}(N)$ acting on half-densities is said to be \textit{classical}, which we will denote by $P\in\Psi_{cl}^{m}(N)$, if its full symbol $\sigma_{P}^{full}$ considered in any coordinate chart is \textit{polyhomogeneous} i.e admits an asymptotic expansion \begin{equation}\sigma_{P}^{full}\left(x,\xi \right)\sim \sum_{j=0}^{+\infty}p_{m-j}(x,\xi)\hspace{0.1cm},\label{eq1}\end{equation} with $p_{m-j}$ homogeneous of degree $m-j$ in $\xi$. (see for example \cite[Sect.18.1]{Horm1} for these notions)
\medskip\noindent\\We will denote $\sigma_{P}$ the principal symbol of $P$ and $sub(P)$ the subprincipal symbol respectively given in any local coordinates by $p_{m}$ and
$$p_{m-1}-\frac{1}{2i}\sum_{j=1}^{n}\frac{\partial^{2}p_{m}}{\partial x_{j}\partial\xi_{j}}\hspace{0.1cm}\cdot$$ 
It is well-known that these quantities are invariant under coordinate changes, we refer once again to \cite{Horm1} and \cite{Horm2} for details.
\\In addition, denote by $H_{\sigma_{P}}$ the symplectic gradient of $\sigma_{P}$ for the canonical symplectic structure on $T^{*}N$ and $\left(\Phi_{t}^{P}\right)$ the Hamiltonian flow generated by $H_{\sigma_{P}}$. \\Integral curves for the vector field $H_{\sigma_{P}}$ are called \emph{bicharacteristic curves} or \emph{bicharacteristics}. Note that since $\sigma_{P}$ is constant along the orbits of $\left(\Phi_{t}^{P}\right)$, the compact submanifold 
$$S^{*}N:=\left\lbrace (x,\xi)\in T^{*}N\hspace{0.1cm};\hspace{0.1cm}\sigma_{P}\left( x,\xi\right)=1\right\rbrace\subset T^{*}N$$ 

\medskip\noindent is invariant by this flow. This is obviously a general fact that Hamiltonian flows preserve the energy level hypersurfaces.\\Thus, we will always consider $\left(\Phi_{t}^{P}\right)$ as a flow acting on $S^{*}N$ and consequently, that the bicharacteristics lie on this hypersurface.

\medskip\noindent \\Now as a first step we propose to generalize the approach of Guillemin, namely \cite[Theorem 3]{Gui}. For $k\in\mathbb{N}$, $A\in\Psi^{1}_{cl}(N)$ and $B\in\Psi^{-k}_{cl}(N)$ self-adjoint operators, we note $U_{0}(t):=e^{itA}$, $U(t):=e^{it(A+B)}$ the unitary groups generated by $A$ and $A+B$ and the so-called \emph{return} operator 
$$R(t):=U(t)U_{0}(-t)\hspace{0.1cm},\hspace{0.1cm}t\in\mathbb{R}\hspace{0.1cm}.$$

\medskip\noindent For each $t\in\mathbb{R}$, \cite[Theorem 1.1]{GD} shows that the Schwartz kernels of $e^{itA}$ and $e^{it(A+B)}$, viewed as distributions on $\mathbb{R}\times N\times N$, are Lagrangian distributions of class $I^{-\frac{1}{4}}\left(\mathbb{R}\times N\times N\hspace{0.05cm}, \hspace{0.05cm}C \right)$ defined by the same canonical relation 

$$\begin{aligned}
C:=\lbrace &(t,\tau),(x,\xi),(y,\eta)\hspace{0.1cm}\mid\hspace{0.1cm}(x,\xi),(y,\eta)\in T^{*}N\setminus\lbrace 0\rbrace\hspace{0.05cm},\\
& (t,\tau)\in T^{*}\mathbb{R}\setminus\lbrace 0\rbrace\hspace{0.05cm},\hspace{0.05cm}\tau+\sigma_{A}(x,\xi)=0\hspace{0.05cm},\hspace{0.05cm}(x,\xi)=\Phi^{A}_{t}(y,\eta)\rbrace\hspace{0.1cm}.   
\end{aligned}$$
\\This implies in particular that $U(t)$ and $U_{0}(t)$ are Fourier integral operator of order $0$ defined by the canonical transformation $\Phi_{t}^{A}$.
\vspace{0.1cm}
\\Then by composition, $R(t)$ defines a Fourier integral operator with the identity as canonical transformation and with polyhomogeneous symbol of order $0$, in other words : $$\forall t\in\mathbb{R},\hspace{0.3cm}R(t)\in\Psi_{cl}^{0}(N)\hspace{0.1cm}.$$

\medskip\noindent The following result is a generalization of \cite[Theorem 3]{Gui} :

\begin{theoreme}\label{th5.1}
    Let $k\in\mathbb{N}^{*}$, $A\in\Psi_{cl}^{1}(N)$ and $B\in\Psi_{cl}^{-k}(N)$ self-adjoint operators. Then $$W(t):=R(t)-Id$$ 
    
    \medskip\noindent is a pseudodifferential operator of order $-k$ on $N$ and its principal symbol is given by : $$\sigma_{W(t)}(x,\xi)=i\int_{\gamma^{t}(x,\xi)}\sigma_{B}\hspace{0.3cm},$$ 
    where $\gamma^{t}(x,\xi):\left[0,t\right]\ni s\mapsto \Phi_{s}^{A}(x,\xi)$ is the bicharacteristic arc of length $t$ passing through $(x,\xi)$.
\end{theoreme}
\begin{proof} 
Differentiating in $t$ the expression of $W(t)$, we get $$\begin{aligned}[t]-i\Dot{W}(t) &=(A+B)e^{it(A+B)}e^{-itA}-e^{it(A+B)}Ae^{-itA} \\
&= (A+B)R(t)-R(t)A\\
&= (A+B)\left(W(t)+Id\right)-\left(W(t)+Id\right)A\\
&= AW(t)-W(t)A+B+BW(t)\hspace{0.1cm},
\end{aligned}$$
where the second equality comes from the fact that $A$ and $e^{-itA}$ commute.
\\In other words, 
\begin{equation}
\left \{
\begin{array}{r @{=} l}
    -i\Dot{W}(t)\hspace{0.1cm} & \hspace{0.1cm}\left[A,W(t)\right]+B+BW(t) \\
          W(0) \hspace{0.1cm}&  \hspace{0.1cm}0
\end{array}
\right.
\label{eq2}    
\end{equation}

\medskip\noindent Then let us note $w(t)$ the full symbol of $W(t)\in\Psi_{cl}^{0}(N)$ in any local coordinates and \begin{equation}
w(t)\sim \sum_{j=0}^{+\infty}w_{-j}(t)\hspace{0.1cm}.\label{eq3}    
\end{equation}

\medskip\noindent If $b$ denotes the full symbol of $B\in\Psi_{cl}^{-k}(N)$ in the same local coordinates, we also note \begin{equation}
b\sim \sum_{j=0}^{+\infty}b_{-k-j}\hspace{0.1cm}.\label{eq4}    
\end{equation}

\medskip\noindent Inserting \eqref{eq3} and \eqref{eq4} into \eqref{eq2} and collecting homogeneous terms with same order $-j$, one can access the leading term in the symbol of $W(t)$ :
\vspace{0.3cm}
\begin{itemize}
    \item[$\bullet$] For $j\in\lbrace 0,...,k-1\rbrace$, we show by induction that $w_{-j}(t)=0$ : First, $w_{0}(t)$ satisfies $$-i\Dot{w}_{0}(t)=i^{-1}H_{\sigma_{A}}w_{0}(t)\hspace{0.1cm},$$ so that 
    $$
\left \{
\begin{array}{r @{=} l}
    \Dot{w}_{0}(t)\hspace{0.1cm} & \hspace{0.1cm}H_{\sigma_{A}}w_{0}(t) \\
          w_{0}(0) \hspace{0.1cm}&  \hspace{0.1cm}0
\end{array}
\right.
$$ 
and then we can write $$w_{0}(t)=\exp\left(tH_{\sigma_{A}}\right)w_{0}(0)= 0\hspace{0.1cm}.$$
Next for any $j\in\lbrace 0,...,k-2\rbrace$, if we assume that $w_{-j}(t)=0$, 
we obtain in the same way $$w_{-j-1}(t)=\exp\left(tH_{\sigma_{A}}\right)w_{-j-1}(0)= 0\hspace{0.1cm}.$$

\item[$\bullet$] For $j= k$, we obtain :
$$-i\Dot{w}_{-k}(t)=i^{-1}H_{\sigma_{A}}w_{-k}(t)+b_{-k}+b_{-k}w_{0}(t)\hspace{0.1cm},$$
and since $w_{0}(t)=0$\hspace{0.1cm},
    $$
\left \{
\begin{array}{r @{=} l}
    \Dot{w}_{-k}(t)\hspace{0.1cm} & \hspace{0.1cm}H_{\sigma_{A}}w_{-k}(t) + ib_{-k}\\
          w_{-k}(0) \hspace{0.1cm}&  \hspace{0.1cm}0
\end{array}
\right.
$$
Then we can apply Duhamel's principle to get $$\begin{aligned}[t]w_{-k}(t)&=\exp\left(tH_{\sigma_{A}}\right)w_{-k}(0)+i\int_{0}^{t}\exp\left((t-s)H_{\sigma_{A}}\right)b_{-k}\hspace{0.1cm}ds\\
&= i\int_{0}^{t}\sigma_{B}\circ\Phi_{s}^{A}\hspace{0.1cm}ds
\end{aligned}$$

\end{itemize}

\end{proof}

\medskip\noindent In the same spirit as \cite[Theorem 4]{Gui} and \cite[Theorem 4.5]{GD} (recall that a version of this theorem in the case of the DN map is given in \textbf{Theorem \ref{Th II.3}}) we can then formulate a trace formula adapted to our problem :

\begin{theoreme}\label{Th IV.2}
Let $k\in\mathbb{N}^{*}$, $A\in\Psi_{cl}^{1}(N)$ and $B\in\Psi_{cl}^{-k}(N)$ such that $A$ and $A+B$ are self-adjoint operators. Assume $A$ is elliptic.
If $A$ only admits a finite number of closed bicharacteristic curves $\gamma_{1},...,\gamma_{r}$ of period $T$ and each of these curves are non-degenerate, then 
   $$S(t):=\mathrm{Tr}\left(U(t)-U_{0}(t)\right)$$ 
   \\is a tempered distribution on $\mathbb{R}$ admitting an isolated singularity in $t=T$ given by a Lagrangian distribution $e_{T}$ with 
   $$\mathrm{Suppsing}\left(e_{T}\right)\subseteq\lbrace T\rbrace$$  
and complete asymptotic expansion $$e_{T}(t)\underset{t\rightarrow T}{\sim} \hspace{0.1cm} \sum_{j=0}^{+\infty}c_{T,k-1+j}\,(t-T)^{k-1+j}\log(t-T+i0)\hspace{0.1cm}.$$ 
The leading coefficient is expressed \footnote{up to a non-zero universal constant depending on $k$.} as 
   
   $$c_{T,k-1}=\sum_{j=1}^{r}\frac{\lvert T_{j}^{\sharp}\rvert e^{im_{j}\frac{\pi}{2}}e^{iT\hspace{0.05cm}\overline{\gamma_{j}}}}{\lvert \det\left(Id-\mathcal{P}_{j}\right)\rvert^{1/2}}\hspace{0.1cm}i\int_{\gamma_{j}}\sigma_{B}\hspace{0.1cm},$$ where $$\overline{\gamma_{j}}:=\frac{1}{T}\int_{\gamma_{j}}\mathrm{sub}(A)\hspace{0.1cm},$$ and $\mathcal{P}_{j}$, $T_{j}^{\sharp}$, $m_{j}$ are respectively the linearized Poincaré map, the primitive period and the Maslov index of $\gamma_{j}$.
\end{theoreme}
\begin{proof}
The proof is similar to that of \cite[Theorem 4]{Gui}, it uses the same arguments as \cite[Theorem 4.5]{GD} on the symbolic calculus for Fourier integral operators. \\For the sake of convenience, let us briefly review the main points in our setting : As we pointed out in the introduction of this section, if we denote $\mu_{0}(t,x,y)$ the Schwartz kernel of $U_{0}(t)$, 
$$\mu_{0}\in I^{-\frac{1}{4}}\left(\mathbb{R}\times N\times N\hspace{0.05cm} ,\hspace{0.05cm} C \right)\hspace{0.1cm}.$$ 
Moreover, by \textbf{Theorem \ref{th5.1}}, $W(t)\in\Psi^{-k}(N)$ and consequently, the distribution 

$$U(t)-U_{0}(t)=W(t)U_{0}(t)$$ 

\medskip\noindent defines a Fourier integral operator ; more precisely its Schwartz kernel $v$ is a Lagrangian distribution of order $-k-\frac{1}{4}$ on $\mathbb{R}\times N\times N$ : 

$$v\in I^{-k-\frac{1}{4}}\left(\mathbb{R}\times N\times N\hspace{0.05cm} ,\hspace{0.05cm} C \right)\hspace{0.1cm}.$$

\medskip\noindent In addition, it is easy to show that 
$S(t)=\pi_{*}\Delta^{*}v\hspace{0.1cm},$
where $\Delta^{*}$ is the pull-back by the diagonal map $\Delta:\mathbb{R}\times N\rightarrow\mathbb{R}\times N\times N$ and $\pi_{*}$ is the push-forward by the projection map $\pi:\mathbb{R}\times N\rightarrow\mathbb{R}$ (see for example \cite[Lemma 6.4]{GD}).
\\Then the theorem follows in the same way as \cite[Section 6]{GD} from \textit{clean intersection theory} developed in \cite[Section 5]{GD}, which can be applied here by assumption on the periodic orbits $\gamma_{1},...,\gamma_{r}$ . \\At this point, it should be noted that, compared to the standard Duistermaat-Guillemin formula, the order of the amplitude is shifted by $k$ in the sense that the corresponding one-dimensional Lagrangian distribution has, near
\(t=T\), an oscillatory integral representation of the form
$$
        e_T(t)
        =
        \int_{-\infty}^{+\infty}e^{i(t-T)\tau}a_T(\tau)\,d\tau
        \quad \bmod C^\infty\hspace{0.1cm};\hspace{0.1cm}
        \quad
        a_T(\tau)\sim \sum_{j=0}^{+\infty}a_{T,k-1+j}\,\tau^{-k-j}\hspace{0.1cm},\hspace{0.1cm}\tau >0
$$
and this homogeneous expansion can be equivalently rewritten as an expansion in the time variable : $$e_{T}(t)\underset{t\rightarrow T}{\sim} \hspace{0.1cm} \sum_{j=0}^{+\infty}c_{T,k-1+j}\,(t-T)^{k-1+j}\log(t-T+i0)\hspace{0.1cm}.$$
\\The leading term $c_{T,k-1}$ in the asymptotic expansion is given, up to a non-zero universal constant (coming from the Fourier transform), by the principal symbol of $\pi_{*}\Delta^{*}v$ and using notations from \cite[Section 4]{Gui}, it can be expressed in terms of the principal symbol $\sigma(v)$ of $v$ as follows :
$$\sigma_{S(t)}=\pi_{*}\Delta^{*}\sigma(v)=\sum_{j=1}^{r}\restriction{\sigma(v)}{C_{j}}\hspace{0.1cm},$$
where

$$C_{j}:=\lbrace(t,\tau),(x,\xi),(y,\eta)\hspace{0.1cm}\mid\hspace{0.1cm}(x,\xi)=(y,\eta)\in\gamma_{j}\hspace{0.05cm},\hspace{0.05cm}t=T\rbrace\subset C\hspace{0.1cm}.$$ 

\medskip\noindent \\Moreover, we note that the principal symbol of $v$ is given by 

$$\sigma(v)\left((t,\tau),(x,\xi),(y,\eta)\right)=\sigma_{W(t)}(x,\xi) \cdot \sigma(\mu_{0})\left((t,\tau),(x,\xi),(y,\eta)\right)\hspace{0.1cm},$$

\medskip\noindent \\where $\sigma(\mu_{0})$ denotes the principal symbol of $\mu_{0}$ .

\medskip\noindent By \cite[Theorem 4.5]{GD}, we know that 
$$\restriction{\sigma(\mu_{0})}{C_{j}}=\frac{\lvert T_{j}^{\sharp}\rvert e^{im_{j}\frac{\pi}{2}}e^{iT\hspace{0.05cm}\overline{\gamma_{j}}}}{\lvert \det\left(Id-\mathcal{P}_{j}\right)\rvert^{1/2}}\hspace{0.1cm}\cdot$$
\\Finally, by \textbf{Theorem \ref{th5.1}}, 
$$\restriction{\sigma_{W(t)}}{C_{j}}=i\int_{\gamma_{j}}\sigma_{B}\hspace{0.1cm},$$ 

\medskip\noindent which conclude the sketch of proof.
\end{proof}

\begin{Rem}
In our case, the Maslov index $m_{j}$ appearing in this result is the Morse index of the periodic geodesic $\gamma_{j}$ (or more precisely here the projection of $\gamma_{j}$) just as in \textbf{Theorem \ref{Th II.3}}. Recall that these indices are zero in the case where the underlying manifold $N$ is Anosov (see \textbf{Remark \ref{Rq II.9}}).
\end{Rem}

In the following, we will use suitable coordinates on the manifold $(M,g)$, called \textit{boundary normal coordinates}. \\If $x':=(x_{1},...,x_{n-1})$ denotes any local coordinates on the boundary near $p\in\partial M$, we can define the associated boundary normal coordinates $(x',x_{n})$ at $p$ by considering $x_{n}$ to be the distance to the boundary along unit speed geodesics normal to $\partial M$. See for example \cite[Section 1]{LU} for more details. \\It follows that in these coordinates, $x_{n}>0$ in $M$ and $\partial M$ is locally characterized by $x_{n}=0$. The metric in these coordinates has the form 

$$g=\sum_{1\leq \alpha,\beta\leq n-1}g_{\alpha\beta}\left(x\right)dx^{\alpha}\otimes dx^{\beta}+ dx^{n}\otimes dx^{n}\hspace{0.1cm}.$$

\medskip\noindent \\Finally, this local coordinate system naturally defines a diffeomorphism 
$$\phi:\mathcal{U}_{p}\rightarrow \mathcal{V}_{p}\hspace{0.1cm},$$
 where $\mathcal{U}_{p}$ is some open set in the half-space 
$$\mathbb{R}^{n}_{+}:=\lbrace (x_{1},...,x_{n})\hspace{0.1cm} ;\hspace{0.1cm} x_{n}\geq 0 \rbrace$$ 
and $\mathcal{V}_{p}:=\phi\left(\mathcal{U}_{p}\right)$ is a neighborhood of $p$ in $M$. For two metrics $g_{1}$ and $g_{2}$ on $M$, let $\phi_{1}$ and $\phi_{2}$ be such diffeomorphisms, then 
$$\phi:=\phi_{1}\circ\phi_{2}^{-1}$$ 
defines a local diffeomorphism sending the unit speed geodesics for $g_{2}$ normal to $\partial M$ onto the unit speed geodesics for $g_{1}$ normal to $\partial M$. 
Recall that this diffeomorphism can be extended globally into a diffeomorphism that is again denoted $\phi$ and that if $(x',x_{n})$ are boundary normal coordinates near a fixed boundary point for the metric $g_{1}$, they are also boundary normal coordinates for the metric $\phi^{*}g_{2}$ . \\This leads to an important remark that we will use in the following to exploit the formulas in \cite{LU} :

\begin{Rem}\label{rq 5.1}
Since $\phi$ is a diffeomorphism of $M$ fixing the boundary we have the well-known gauge invariance 
$$\Lambda_{\phi^{*}g_{2}}=\Lambda_{g_{2}}$$ 

\medskip\noindent and then in particular, assuming that $g_{2}$-boundary normal coordinates are also $g_{1}$-boundary normal coordinates does not change the DN map $\Lambda_{g_{2}}$.
\\For our problem in the conformal case and more precisely in \textbf{Lemma \ref{lem 5.1}} below, we will therefore consider the normalized metric
$$
\widetilde{g}:=\phi^*(cg)\hspace{0.1cm}.
$$
Although \(\widetilde{g}\) is no longer exactly conformal to \(g\), the fact that $\restriction{c}{\partial M} \equiv 1$ and the inductive assumption at order $j\geq 1$

$$
        \partial_n^m c|_{x_n=0}=0\hspace{0.2cm},\hspace{0.2cm}\forall \hspace{0.1cm}1\leq m\leq j$$
in \(g\)-boundary normal coordinates implies $$c=1+O(x_n^{j+1})\hspace{0.1cm},$$
where the notation $O(x_n^{k})$ denotes an expression of the form $x_n^{k}F_k(x',x_n)$ near the boundary with \(F_k\) smooth up to the boundary. In particular, in \(g\)-boundary normal coordinates one has
$$
        (cg)_{n\alpha}=0\quad,
        \quad
        (cg)_{nn}=1+O(x_n^{j+1})\hspace{0.1cm}.
$$
Thus, $g$-boundary normal coordinates fail to be exact \(cg\)-boundary normal coordinates
only at order \(j+1\) in the normal component and the diffeomorphism $\phi$ actually changes the tangential part of the metric only by \(O(x_n^{j+2})\) :
$$
        \widetilde g_{\alpha\beta}
        =
        cg_{\alpha\beta}
        +
        O(x_n^{j+2})
$$
and therefore for the inverse metric
$$
        \widetilde g^{\alpha\beta}
        =
        c^{-1}g^{\alpha\beta}
        +
        O(x_n^{j+2})\hspace{0.1cm}.
$$ 
\end{Rem}

We now have the tools we need to obtain the desired boundary determination results :

\medskip\noindent First, \textbf{Corollary \ref{coro 1.3.1}} shows that if the total symbols of $\Lambda_{cg}$ and $\Lambda_{g}$ coincide at principal and subprincipal order, these two operators differ only by a pseudodifferential operator of order $-1$. So in this case, as already mentioned, we are now able to combine \textbf{Theorem \ref{Th IV.2}} with the injectivity of the X-ray transform on functions (\textbf{Theorem \ref{Th II.5}}) to access the homogeneous terms of orders $j\leq -1$ contained in the symbol of $\Lambda_{cg}-\Lambda_{g}\in\Psi^{-1}(\partial M)$ and, more precisely, to show recursively that they are all zero if the metrics $g$ and $cg$ are Steklov isospectral. 
\vspace{0.1cm}
\medskip\noindent \\This is the subject of the next theorem and, following the work of Lee and Uhlmann \cite[Proposition 1.3]{LU}, we begin with a lemma specifying these homogeneous terms of order $j\leq -1$ :

\begin{lem}\label{lem 5.1}
Let $(M,g)$ be a smooth compact Riemannian manifold of dimension $n\geq 2$ with smooth boundary and $c:M\rightarrow \mathbb{R}^{+}_{*}$ be a smooth function satisfying $\restriction{c}{\partial M} \equiv 1$.
Let $$\sigma^{full}_{\Lambda_{cg}}\sim \sum_{j=-1}^{+\infty}\widetilde{a}_{-j}$$
and $$\sigma^{full}_{\Lambda_{g}}\sim \sum_{j=-1}^{+\infty}a_{-j}$$ 
be respectively the full symbols of $\Lambda_{cg}$ and $\Lambda_{g}$ considered in any local coordinate system 
$x':=(x_{1},...x_{n-1})$ on the boundary. 

\medskip\noindent In the $g$-boundary normal coordinates $(x',  x_{n})$ :

\medskip\noindent For any $j\geq 1$, if 
$$\restriction{\partial_{n}^{m}c}{x_{n}= 0}= 0\hspace{0.2cm},\hspace{0.2cm}\forall \hspace{0.1cm}1\leq m\leq j\hspace{0.2cm},$$
then
$$\widetilde{a}_{-j}-a_{-j}=\frac{\left(n-2\right)}{4}\left(-\frac{1}{2\lvert\xi'\rvert_{g^{\circ}}}\right)^{j}\restriction{\partial_{n}^{j+1}c}{x_{n}= 0}\hspace{0.1cm}.$$
\end{lem}
\begin{proof}
Let $x':=(x_{1},...,x_{n-1})$ be any local coordinate system on the boundary and $(x',x_{n})$ the associated $g$-boundary normal coordinates. By \textbf{Remark \ref{rq 5.1}}, we have to replace
\(cg\) by
\[
        \widetilde g:=\phi^*(cg)\quad,
        \quad \phi|_{\partial M}=\mathrm{Id}_{\partial M}\hspace{0.1cm},
\]
so that $g$-boundary normal are also $\widetilde g$-boundary normal coordinates. Recall that
$$
        \Lambda_{\widetilde g}=\Lambda_{cg}\hspace{0.1cm},
$$
so the full symbol of \(\Lambda_{cg}\) may be computed from that of $\Lambda_{\widetilde g}$ in these coordinates.

Using \cite[Proposition 1.3]{LU}, we denote by $(g^{\alpha\beta})$ the inverse matrix of $(g_{\alpha\beta})$ and introduce for any $1\leq\alpha ,\beta\leq n-1$, 

$$h^{\alpha\beta}:=\partial_{n}g^{\alpha\beta}\hspace{0.2cm},\hspace{0.2cm}h=\sum_{1\leq a ,b\leq n-1}g_{ab}h^{ab}\hspace{0.1cm},$$
$$\widetilde{h}^{\alpha\beta}:=\partial_{n}\widetilde{g}^{\alpha\beta}\hspace{0.2cm},\hspace{0.2cm}\widetilde{h}=\sum_{1\leq a ,b\leq n-1}\widetilde{g}_{ab}\widetilde{h}^{ab}$$
and the quadratic forms $$k^{\alpha\beta}:=h^{\alpha\beta}-hg^{\alpha\beta}\hspace{0.2cm},\hspace{0.2cm}\widetilde{k}^{\alpha\beta}:=\widetilde{h}^{\alpha\beta}-\widetilde{h}\widetilde{g}^{\alpha\beta}$$ respectively associated with the metrics $g$ and $\widetilde g$. 
\\Now suppose that for any $j\geq 1$,
$$\restriction{\partial_{n}^{m}c}{x_{n}= 0}= 0\hspace{0.2cm},\hspace{0.2cm}\forall \hspace{0.1cm}1\leq m\leq j\hspace{0.2cm}.$$
By \textbf{Remark \ref{rq 5.1}},
\[
        \widetilde{g}^{\alpha\beta}
        =
        c^{-1}g^{\alpha\beta}
        +
        O(x_n^{j+2})
\]
and then
\[
        \widetilde h^{\alpha\beta}
        =
        \partial_n(c^{-1}g^{\alpha\beta})
        +
        O(x_n^{j+1})\hspace{0.1cm},
\]
so that 
$$
        \widetilde h=
        \widetilde g_{ab}\widetilde h^{ab}=
        c\,g_{ab}\,
        \partial_n(c^{-1}g^{ab})
        +
        O(x_n^{j+1})\hspace{0.1cm}.
$$
Consequently, we have
\[
\begin{aligned}
        \widetilde k^{\alpha\beta}
        &=
        \partial_n(c^{-1}g^{\alpha\beta})
        -
        \left(
        c\,g_{ab}\partial_n(c^{-1}g^{ab})
        \right)c^{-1}g^{\alpha\beta}
        +
        O(x_n^{j+1})\hspace{0.1cm},
\end{aligned}
\]
with
\[
        \partial_n(c^{-1}g^{\alpha\beta})
        =
        c^{-1}h^{\alpha\beta}
        -
        \frac{\partial_{n}c}{c^{2}}g^{\alpha\beta}\hspace{0.1cm}.
\]
Furthermore,
\[
\begin{aligned}
        c\,g_{ab}\partial_n(c^{-1}g^{ab})
        &=
        c\,g_{ab}
        \left(
        c^{-1}h^{ab}
        -
        \frac{\partial_{n}c}{c^{2}}g^{ab}
        \right) \\
        &=
        h-(n-1)\frac{\partial_{n}c}{c}
\end{aligned}
\]
and then
\[
\begin{aligned}
        \widetilde k^{\alpha\beta}
        &=
        c^{-1}h^{\alpha\beta}
        -
        \frac{\partial_{n}c}{c^{2}}g^{\alpha\beta} \\
        &\quad
        -
        \left(
        h-(n-1)\frac{\partial_{n}c}{c}
        \right)c^{-1}g^{\alpha\beta}
        +
        O(x_n^{j+1}) \\
        &=
        c^{-1}k^{\alpha\beta}
        +
        (n-2)\frac{\partial_{n}c}{c^{2}}g^{\alpha\beta}
        +
        O(x_n^{j+1})\hspace{0.1cm}.
\end{aligned}
\]

Since
\[
        c|_{x_n=0}=1\quad,
        \quad
        \partial_n^m c|_{x_n=0}=0\hspace{0.1cm},
        \quad
        1\leq m\leq j\hspace{0.1cm},
\]
we get
\[
        \partial_n^j
        \left(
        c^{-1}k^{\alpha\beta}
        \right)\Big|_{x_n=0}
        =
        \partial_n^j k^{\alpha\beta}\Big|_{x_n=0}\hspace{0.1cm},
\]
and
\[
        \partial_n^j
        \left(
        \frac{\partial_{n}c}{c^{2}}g^{\alpha\beta}
        \right)\Big|_{x_n=0}
        =
        \partial_n^{j+1}c\Big|_{x_n=0}\,
        g^{\alpha\beta}\Big|_{x_n=0}\hspace{0.1cm}.
\]
It should be noted that when we apply \(\partial_n^j\) and restrict ourselves to \(x_n=0\), the remainder
\(O(x_n^{j+1})\) makes no contribution, which finally gives us 
\[
        \partial_n^j\widetilde k^{\alpha\beta}\Big|_{x_n=0}
        =
        \partial_n^j k^{\alpha\beta}\Big|_{x_n=0}
        +
        (n-2)\partial_n^{j+1}c\Big|_{x_n=0}\,
        g^{\alpha\beta}\Big|_{x_n=0}\hspace{0.1cm}.
\]
By \cite[Proposition 1.3]{LU}, for \(j\geq 1\),
\[
        \widetilde a_{-j}
        =
        \left(
        -\frac{1}{2|\xi'|_{g^\circ}}
        \right)^{j+2}
        \sum_{1\leq\alpha,\beta\leq n-1}
        \left(
        \partial_n^j\widetilde{k}^{\alpha\beta}
        \right)
        \xi_\alpha\xi_\beta
        +
        T_{-j}(\widetilde g_{\alpha\beta})\hspace{0.1cm},
\]
where \(T_{-j}(\widetilde g_{\alpha\beta})\) only involves the boundary values of \(\widetilde g_{\alpha\beta}\), \(\widetilde g^{\alpha\beta}\), and
their normal derivatives of order at most \(j\). \\Obviously we also have the same formula for $a_{-j}$, just replacing $\widetilde{k}^{\alpha\beta}$ by $k^{\alpha\beta}$ and $\widetilde g$ by $g$ :
\[
        a_{-j}
        =
        \left(
        -\frac{1}{2|\xi'|_{g^\circ}}
        \right)^{j+2}
        \sum_{1\leq\alpha,\beta\leq n-1}
        \left(
        \partial_n^j k^{\alpha\beta}
        \right)
        \xi_\alpha\xi_\beta
        +
        T_{-j}(g_{\alpha\beta})\hspace{0.1cm}.
\]
In addition, since
\[
        \partial_n^m\widetilde g_{\alpha\beta}\big|_{x_n=0}
        =
        \partial_n^m(cg_{\alpha\beta})\big|_{x_n=0}\quad,
        \quad
        0\leq m\leq j+1\hspace{0.1cm},
\]
and
\[
        \partial_n^m c|_{x_n=0}=0\hspace{0.1cm},
        \quad
        1\leq m\leq j\hspace{0.1cm},
\]
we have equality of the remainders 
\[
        T_{-j}(\widetilde g_{\alpha\beta})
        =
        T_{-j}(g_{\alpha\beta})\hspace{0.1cm}.
\]
Thus, noticing that 
$$\sum_{1\leq \alpha,\beta \leq n-1}g^{\alpha\beta}\xi_{\alpha}\xi_{\beta}=\lvert\xi'\rvert_{g^{\circ}}^{2}$$
and noting $$T_{-j}:=\sum_{1\leq \alpha,\beta \leq n-1}T_{-j}(g_{\alpha\beta})\hspace{0.1cm},$$
we get for $j\geq 1$,
$$\begin{aligned}[t] \widetilde{a}_{-j} &=\left(-\frac{1}{2\lvert\xi'\rvert}\right)^{j+2}\sum_{1\leq \alpha,\beta \leq n-1}\left[(n-2)\left(\restriction{\partial_{n}^{j+1}c}{x_{n}= 0}\right)g^{\alpha\beta}+\partial_{n}^{j}k^{\alpha\beta}\right]\xi_{\alpha}\xi_{\beta} + T_{-j}\\
&= \frac{\left(n-2\right)}{4\lvert\xi'\rvert_{g^{\circ}}^{2}}\left(-\frac{1}{2\lvert\xi'\rvert_{g^{\circ}}}\right)^{j}\restriction{\partial_{n}^{j+1}c}{x_{n}= 0}\sum_{1\leq \alpha,\beta \leq n-1}g^{\alpha\beta}\xi_{\alpha}\xi_{\beta}\hspace{0.1cm}+ \hspace{0.1cm}a_{-j}\\
&= \frac{\left(n-2\right)}{4}\left(-\frac{1}{2\lvert\xi'\rvert_{g^{\circ}}}\right)^{j}\restriction{\partial_{n}^{j+1}c}{x_{n}= 0} \hspace{0.1cm}+ \hspace{0.1cm}a_{-j}\hspace{0.1cm}.
\end{aligned}$$ 
This proves the lemma. 

\end{proof}

\medskip\noindent We deduce the boundary determination of the jet of the conformal factor $c$ at any point $p\in\partial M$ :
\begin{theoreme}\label{Th IV.6}
    Let $(M,g)$ be a smooth compact Riemannian manifold of dimension $n\geq 3$ with smooth boundary. Assume that $(\partial M,g^{\circ})$ is Anosov with simple length spectrum. Let $c:M\rightarrow \mathbb{R}^{+}_{*}$ be a smooth function satisfying $\restriction{c}{\partial M} \equiv 1$.
    \\If 
    $$\mathrm{spec}(\Lambda_{cg})=\mathrm{spec}(\Lambda_{g})\hspace{0.1cm},$$ 

\medskip\noindent then the boundary normal derivatives of the conformal factor are zero :
$$\restriction{\partial_{\nu}^{j}c}{\partial M}=0\hspace{0.1cm},\hspace{0.1cm}\forall j\geq 1\hspace{0.1cm}.$$
Consequently, the Taylor series of $c$ at any point $p\in\partial M$ in any local coordinates are reduced to the constant function identically equal to $1$.
\end{theoreme}
\begin{proof}
Let us consider $x':=(x_{1},...x_{n-1})$ any local coordinate system on the boundary and $(x',  x_{n})$ the associated $g$-boundary normal coordinates near $\partial M$.
\\We proceed by induction on the order of derivation $j\geq 1$ : 
\vspace{0.3cm}
\begin{itemize}
    \item[$\bullet$] $j=1$ :  By \textbf{Theorem \ref{Th III.18}}, $$\restriction{\partial_{n}c}{x_{n}= 0} = 0$$ since $(\partial M,g^{\circ})$ is Anosov with simple length spectrum.
    
    \vspace{0.2cm}
    
    \item[$\bullet$] Let $j\geq 1$ s.t for all $1\leq m\leq j$, 
    $$\restriction{\partial_{n}^{m}c}{x_{n}= 0}=0\hspace{0.1cm}.$$ 
Then using \textbf{Lemma \ref{lem 5.1}}, 
    $$\Lambda_{cg}-\Lambda_{g}\in\Psi^{-j}(\partial M)\hspace{0.1cm}.$$
Now we can apply \textbf{Theorem \ref{Th IV.2}} to $(N,h)=(\partial M,g^{\circ})$,
$$A=\Lambda_{g}\in\Psi^{1}(\partial M)\hspace{0.2cm}\text{and}\hspace{0.2cm} B=\Lambda_{cg}-\Lambda_{g}\in\Psi^{-j}(\partial M)\hspace{0.2cm}:$$ 
Recalling that 
    \vspace{0.2cm}
    \begin{itemize}
        \item[-] $\mathrm{spec}(\Lambda_{cg})=\mathrm{spec}(\Lambda_{g})\hspace{0.1cm}.$
        \vspace{0.1cm}
        \item[-] $\restriction{c}{\partial M} \equiv 1\hspace{0.1cm}.$
        \vspace{0.1cm}
        \item[-] $(\partial M,g^{\circ})$ is Anosov with simple length spectrum (which implies, in \textbf{Theorem \ref{Th IV.2}}, $r=1$ and the non-degeneracy of closed geodesics as we pointed out in Section \ref{Subsection II.2}).
    \end{itemize} 
    
    \medskip\noindent we obtain that the leading term in the expansion of $S(t)$ is zero, and then using \textbf{Theorem \ref{Th IV.2}}, we get that for each periodic bicharacteristic $s\mapsto\Phi_{s}^{A}\left(x,\xi\right)$, $$0=\int_{0}^{\ell}\sigma_{B}\left(\Phi_{s}^{A}\left(x,\xi\right)\right)ds\hspace{0.1cm},$$ 

\medskip\noindent where $\ell$ obviously denotes the length of the bicharacteristic. Given the expression for the principal symbol of $B\in\Psi^{-j}(\partial M)$ in \textbf{Lemma \ref{lem 5.1}}, we conclude as usual by injectivity of the geodesic $X$-ray transform on functions (\textbf{Theorem \ref{Th II.5}}), that 
    $$\restriction{\partial_{n}^{j+1}c}{x_{n}= 0}=0\hspace{0.1cm}.$$
\end{itemize}
\vspace{0.2cm}
It is then easy to see that the Taylor series of the metrics $g$ and $cg$ at any $p\in\partial M$ are identical in $g$-boundary normal coordinates $(x',x_{n})$ near this point :  
\vspace{0.1cm}
\\By simply using \textbf{Remark \ref{rq 5.1}} and \cite[Proposition 1.3]{LU}, we know that the knowledge of the full symbols of $\Lambda_{cg}$ and $\Lambda_{g}$ determine, in $g$-boundary normal coordinates the Taylor series of $cg$ and $g$ respectively. \\But with the same notations as in \textbf{Lemma \ref{lem 5.1}}, we have just shown that in these coordinates,
\begin{equation}
\widetilde{a}_{-j}-a_{-j}=0\hspace{0.2cm},\hspace{0.2cm}\forall j\geq -1\hspace{0.2cm}.\label{eq6}    
\end{equation}

\medskip\noindent Taylor series of $c$ at $p$ are then reduced to the constant function identically equal to $1$ in any local coordinates since the nullity of derivatives is obviously an invariant property under coordinate changes.
\end{proof}

\begin{Rem}
Note that \eqref{eq6} is equivalent to saying that $\Lambda_{cg}$ and $\Lambda_{g}$ coincide modulo a smoothing operator.     
\end{Rem}

\medskip\noindent By a simple argument of analytic continuation on the function $c$, we then obtain that if $c$ is real-analytic and $M$ is connected, then $c\equiv 1$ on $M$.

\section{Further results and remarks}\label{Section V}

Here we discuss other similar inverse problems which the author believes could be interesting to study in the line of this paper. In particular, this section establishes \textbf{Theorem \ref{Th I.6}} using the method developed in Section \ref{Section IV}. The result deals with the Steklov spectral inverse problem of recovering a potential $q$ from the Steklov spectrum of a smooth compact Riemannian manifold $M$ of dimension $n\geq 3$ with boundary. 
\\This inverse problem involves the Schrödinger operator $\mathcal{L}_{g,q}:=-\Delta_{g}+q$ where $q$ is a smooth real-valued function on $M$ : It is well known that if $0$ is not a Dirichlet eigenvalue of $\mathcal{L}_{g,q}$ in $M$, then the Dirichlet problem 
$$\left\{\begin{aligned}
 \mathcal{L}_{g,q}v &= 0 \\
\restriction{v}{\partial M} &= f \\
\end{aligned}\right.$$
has a unique solution in $H^{1}(M)$ for any $f\in H^{1/2}(\partial M)$ and we can define the associated DN map 
$$\Lambda_{g,q}: f\mapsto\restriction{\partial_{\nu}u}{\partial M}\hspace{0.1cm}.$$

\medskip\noindent We refer, for example, to \cite{Sal} for details. In the same way as in the case of the standard Steklov operator $\Lambda_{g}$ presented in the Introduction, $\Lambda_{g,q}$ is a self-adjoint elliptic and classical (in the sense of $\left(\ref{eq1}\right)$)  pseudodifferential operator of order $1$ on the closed manifold $\partial M$ (see, for example, \cite{DDSF} or \cite{Cek}).
\\In particular, its spectrum is discrete and is given by a sequence of eigenvalues
$$\sigma_{0}<\sigma_{1}\leq\sigma_{2}\leq...\rightarrow \infty\hspace{0.1cm}.$$
Using the expressions for the full symbols of $\Lambda_{g,q}$ in boundary normal coordinates, computed in \cite[Lemma 8.6]{DDSF},  we can easily show a result analogous to \textbf{Lemma \ref{lem 5.1}} for this problem :

\begin{lemme}
Let $q_{1}$ and $q_{2}$ two smooth potentials on $M$.
Let $$\sigma^{full}_{\Lambda_{g,q_{2}}}\sim \sum_{j=-1}^{+\infty}\widetilde{a}_{-j}$$
and $$\sigma^{full}_{\Lambda_{g,q_{1}}}\sim \sum_{j=-1}^{+\infty}a_{-j}$$ 

\medskip\noindent be respectively the full symbols of $\Lambda_{g,q_{2}}$ and $\Lambda_{g,q_{1}}$ considered in any local coordinate system 
$x':=(x_{1},...x_{n-1})$ on the boundary. 

\medskip\noindent In the $g$-boundary normal coordinates $(x',  x_{n})$, 

$$\widetilde{a}_{1}-a_{1}=0=\widetilde{a}_{0}-a_{0}$$

\medskip\noindent and
$$\widetilde{a}_{-1}-a_{-1}=\restriction{\left(q_{2}-q_{1}\right)}{x_{n}= 0}\hspace{0.1cm}.$$

\medskip\noindent Furthermore, for all $j\geq 2$, if 

$$\restriction{\partial_{n}^{m-2}\left(q_{2}-q_{1}\right)}{x_{n}= 0}\hspace{0.1cm},\hspace{0.1cm}\forall j\geq m\geq 2\hspace{0.1cm},$$

\medskip\noindent then in the $g$-boundary normal coordinates $(x',  x_{n})$ :
$$\widetilde{a}_{-j}-a_{-j}=\left(-\frac{1}{2\lvert\xi'\rvert_{g^{\circ}}}\right)^{j}\restriction{\partial_{n}^{j-1}\left(q_{2}-q_{1}\right)}{x_{n}= 0}\hspace{0.1cm}.$$
\end{lemme}

\medskip\noindent Note that the proof of this result is actually easier than \textbf{Lemma \ref{lem 5.1}} since all the terms involving the metric are identical for both operators. Note also that the identity at homogeneity order $-1$ can be seen directly here via the expression of the principal symbol given in \textbf{Proposition \ref{Prop B.3}}.
\\Now, using the same arguments developed in Section \ref{Section IV}, we can show a result similar to \textbf{Theorem \ref{Th I.1}}, namely that the Steklov spectrum determines the potential in the following sense :

\begin{Theorem}\label{Th V.2}
    Assume that $(\partial M,g^{\circ})$ is Anosov with simple length spectrum. Let $q_{1}$ and $q_{2}$ two smooth potentials on $M$.
\\If the Steklov spectra agree
$$\mathrm{spec}(\Lambda_{g,q_{1}})=\mathrm{spec}(\Lambda_{g,q_{2}})\hspace{0.1cm},$$ 
then the boundary normal derivatives at any order coincide :
$$\restriction{\partial_{\nu}^{j-1}\left(q_{2}-q_{1}\right)}{\partial M}=0\hspace{0.1cm},\hspace{0.1cm}\forall j\geq 1\hspace{0.1cm}.$$
Consequently, the Taylor series of $q_{1}$ and $q_{2}$ at each point $p\in\partial M$ in any local coordinates are equal and 
    $$\Lambda_{g,q_{1}}= \Lambda_{g,q_{2}}$$ 
modulo a smoothing operator.
    
\end{Theorem}
\begin{proof}
The proof follows exactly the same ideas as \textbf{Theorem \ref{Th IV.6}} : We proceed by induction on the order of derivation $j\geq 0$ and apply at each step the trace formula \textbf{Theorem \ref{Th IV.2}} to 
$$(N,h)=(\partial M,g^{\circ})\hspace{0.2cm},\hspace{0.2cm}A=\Lambda_{g,q_{1}}\in\Psi^{1}(\partial M)\hspace{0.2cm}\text{,}\hspace{0.2cm} B=\Lambda_{g,q_{2}}-\Lambda_{g,q_{1}}\in\Psi^{-j-1}(\partial M)$$ 
which, combined with the isospectrality assumption, gives us for any periodic bicharacteristic $s\mapsto\Phi_{s}^{A}\left(x,\xi\right)$, 
$$0=\int_{0}^{\ell}\sigma_{B}\left(\Phi_{s}^{A}\left(x,\xi\right)\right)ds\hspace{0.1cm},$$ 

\medskip\noindent where $\ell$ denotes the length of the bicharacteristic.

\medskip\noindent Then combining the expression for the principal symbol of $B\in\Psi^{-j-1}(\partial M)$ given by the previous lemma with the injectivity of the X-ray transform (\textbf{Theorem \ref{Th II.5}}), we obtain in the $g$-boundary normal coordinates $(x',x_{n})$,
$$\restriction{\partial_{n}^{j}\left(q_{2}-q_{1}\right)}{x_{n}=0}=0\hspace{0.1cm},$$

\medskip\noindent equality that is again completely invariant under coordinate changes, of course.
\\Using the notations from the previous lemma, we have just shown that,
$$\widetilde{a}_{-j}-a_{-j}=0\hspace{0.2cm},\hspace{0.2cm}\forall j\geq -1$$
or equivalently that $\Lambda_{g,q_{1}}$ and $\Lambda_{g,q_{2}}$ coincide modulo a smoothing operator. \\To conclude, we can then simply use the calculations of \cite[Section 8]{DDSF} or \cite[Section 3.2]{Cek} which show that in $g$-boundary normal coordinates the Taylor series of $q_{1}$ and $q_{2}$ at any point $p\in\partial M$ are determined by the knowledge of the full symbols of $\Lambda_{g,q_{1}}$ and $\Lambda_{g,q_{2}}$ respectively.
\end{proof}

\begin{Rk}
Note that initialization here is much simpler, since it is obtained directly by the same argument as at higher orders, unlike the case treated in this paper. In the latter case, initialization consisted in showing  that the normal derivative of the conformal factor at the boundary is zero, but this quantity appears at the subprincipal level (via the trace formula \textbf{Theorem \ref{Th II.3}}) and then is not directly covered by our method.  For this reason, this non-linear case was also treated separately in Section \ref{Section III}; the result was stated in \textbf{Theorem \ref{Th III.18}}.
\end{Rk}

To complete the above remark, let us note that this problem is actually made completely linear by the fact that the operators $\Lambda_{g,q_{1}}$ and $\Lambda_{g,q_{2}}$ already coincide at principal and subprincipal order and that the information on the potentials is given from order $-1$, which allows us to apply \textbf{Theorem \ref{Th IV.2}} directly. In this sense, the problem is simpler than the case treated in this paper, for which obtaining information in the principal order is a non-linear problem and requires the addition of relatively strong geometric assumptions as discussed in Appendix \ref{App A}. \\Finally, let us note that just as \textbf{Theorem \ref{Th IV.2}} is an analogue of \cite[Theorem 4]{Gui}, its application \textbf{Theorem \ref{Th V.2}} can be compared with \cite[Theorem 2]{GK} on the recovery of a potential $q$ from the spectrum of the associated Schrödinger operator. Moreover, it seems worth mentioning here that \textbf{Theorem \ref{Th IV.2}} also gives us in particular, and this is actually equivalent to \textbf{Theorem \ref{Th V.2}}, an analogue of \cite[Theorem 3(a)]{GK}, namely :

\begin{Theorem}
Let $(M,g)$ be a compact smooth Riemannian manifold with boundary and $q$ be a smooth potential on $M$. Assume that $(\partial M,g^{\circ})$ is Anosov with simple length spectrum. Then the spectrum of $\Lambda_{g,q}$ determines the integrals of $q$ and all its normal derivatives at the boundary over the periodic geodesics of $g^{\circ}$.
\end{Theorem}
 
A natural question raised by \textbf{Theorem \ref{Th V.2}} is whether we can explicitly construct an example of Steklov isospectral potentials which are not the same, obviously when the simple length spectrum condition is not verified.
\\In the case of the Schrödinger operator $\mathcal{L}_{g,q}=-\Delta_{g}+q$, such examples are shown in \cite{Br} via the general method of Sunada \cite{Sun}, which remains probably the most widely used and cited technique to construct isospectral manifolds. \\We refer to \cite{Gord} for details and more specifically to \cite{GordWebb} for the Steklov spectrum.
To the author's knowledge, there is no reference for the operator $\Lambda_{g,q}$, but it may be known in folklore that the Sunada construction can be adapted again.
\vspace{0.2cm}

In another direction, recall that in this paper we crucially exploit the fact that the Steklov operator has a non-zero subprincipal symbol which enabled us to use the additional information provided by the principal invariant of the wave trace \textbf{Theorem \ref{Th II.3}}. \\In this way, in \textbf{Theorem \ref{Th III.18}}, we have recovered the normal derivative of the conformal factor at the boundary which appears as an attenuation factor in a transport-type equation. (see \textbf{Theorem \ref{Th III.17}})
\\To continue in this spirit, we can try to exhibit inverse spectral problems for which the operator has a non-zero subprincipal symbol and which reveal relevant information but more complicated in nature than a simple function, for example a $1$-form.
\\One can think of the magnetic Laplacian (which is a first order perturbation of the usual Laplacian) or, in the case of the Steklov inverse problem, the magnetic Dirichlet-to-Neumann map $\Lambda_{g,A,q}$, which roughly speaking consists of adding a magnetic potential i.e a smooth $1$-form $A$ on $M$. More precisely, this is the Steklov operator associated with the Dirichlet problem for the Schrödinger operator 
$$\mathcal{L}_{g,A,q}:=d_{A}^{*}d_{A}+q\hspace{0.1cm},$$ 
where $d_{A}:=d+iA$ denotes the \textit{magnetic differential} and $d_{A}^{*}:=d^{*}-i\langle A^{\sharp}, \cdot
\rangle$ its formal adjoint (see, for example, \cite{DDSF} or the recent \cite{ChaGitHabPey} for details).
\\We hope to discuss these cases in a forthcoming article, generalizing \textbf{Theorem \ref{Th III.17}} by the same token.
Indeed, Paternain's argument \cite{Pat} using Gysin sequences still works to get a satisfactory statement but a new formalism in the spirit of Pestov identity  can be introduced to handle this equation.
\vspace{0.3cm}

In the problems discussed above in this section, the metric is fixed and the aim is to recover a potential. \\But we can also try to generalize the results obtained in this paper in the case where the metrics are no longer in the same conformal class.
\\This problem obviously poses significant technical issues, since we are no longer reasoning about a function (the conformal factor) but directly about the metric tensor, and it is not clear that we can obtain satisfactory results in the non-linear case (when isospectral deformations are not considered).
\\More precisely, consider two metrics $g_{1}$ and $g_{2}$ and assume that $(\partial M,g_{i}^{\circ})$, $i\in\lbrace 1,2\rbrace$ is Anosov with simple length spectrum. Then if the Steklov spectrum is preserved
$$\mathrm{spec}(\Lambda_{g_{1}})=\mathrm{spec}(\Lambda_{g_{2}})\hspace{0.1cm},$$
can we show that the Taylor series of $g_{1}$ and $g_{2}$ in boundary normal coordinates are equal at each point on the boundary and that $\Lambda_{g_{1}}$ and $\Lambda_{g_{2}}$ coincide modulo a smoothing operator ? 
\\Of course, the answer is a priori no, at least without further assumptions.
\\Indeed, as in the case of conformal metrics, information at principal order poses a significant problem, and therefore it seems necessary here to assume at least that the two metrics are isometric at the boundary, i.e. there is a diffeomorphism $\varphi:\partial M\rightarrow\partial M$ such that
$$g_{2}^{\circ}=\varphi^{*}g_{1}^{\circ}\hspace{0.1cm}.$$

\medskip\noindent We can also note that, at subprincipal order, the relevant quantity that appears is the second fundamental form of the boundary, since for any metric $g$, 

$$sub(\Lambda_{g})\left(x,\xi\right)=\dfrac{1}{2}\left(\dfrac{\mathcal{L}_{g}(\xi,\xi)}{\lvert\xi\rvert^{2}}-H_{g}\right)\hspace{0.1cm},$$

\medskip\noindent where $\mathcal{L}_{g}(\xi,\xi)$ is the second fundamental form of $\partial M$ in $(M,g)$ and $H_{g}$ its trace, also called \textit{mean curvature} of $M$. We refer to \cite[Chapter 12, Proposition C.1]{Tay} for this standard result.\\ Then we can try to mimick the conformal case and see if we are able to get usable information as well and obtain similar results. In addition, the problem would no longer require the injectivity of the X-ray transform on functions, but rather an injectivity property for the X-ray transform on symmetric tensors of order $2$ which, for closed Anosov manifolds, is only known in dimension $2$ (see \cite[Theorem 1.1]{PatSalUhl2} or more generally \cite[Theorem 1.4]{Guillarmou} for tensors of any order), or in any dimension but with non-positive sectional curvature (see \cite[Theorem 1.3]{CrSh}). However, we should mention here that the X-ray transform of symmetric tensors of arbitrary degree is generically injective on closed Anosov manifolds of dimension $n\geq 3$, as shown very recently by Ceki{\'c} and Lefeuvre in \cite{CekiLef}.
\vspace{0.2cm}

To conclude this paper, note that we may also be interested in establishing stability estimates to complement the results of this paper, even if the techniques used here appear to be non-quantitative. In another direction it would also be interesting to try to generalize some known compactness results for sets of Laplace isospectral metrics, in particular \cite[Theorem 2]{BrPeYa} in the case of conformal metrics, to the Steklov spectrum, this could be done by exploiting heat invariants if possible.

\appendix

\section{A few results on the principal order}\label{App A}

Here we establish some simple results on the spectral inverse problem at principal order ; more precisely we give examples where the Steklov isospectrality of two conformal metrics $g$ and $cg$ implies their equality at the boundary i.e 
$$\restriction{c}{\partial M} \equiv 1\hspace{0.1cm}.$$
As mentioned in this paper, this highly non-linear problem cannot be treated with the techniques used here and it is of course not true in general just as the analogue for the Laplacian is not. \\Consequently, we are only able to show this result for the linear problem (i.e considering isospectral deformations) and in the non-linear case under much stronger geometrical assumptions due to the lack of exploitable spectral information.
Indeed, for the latter, we only exploit the boundary volume as a spectral invariant, given by the well-known Weyl's law (see \textbf{Proposition \ref{Prop II.1}}).
\vspace{0.2cm}
\\To begin with, using the approach developped at the end of Section \ref{Subsection II.2}, inspired by that of \cite{GK}, it is easy to show that $(\partial M,g^{\circ})$ is infinitesimally spectrally rigid in the following sense :

\begin{prop}\label{Prop III.1}
Let $(M,g)$ be a smooth compact Riemannian manifold with boundary. Assume that $g^{\circ}$ is Anosov with simple length spectrum. Let $(g_{s})_{0\leqslant s\leqslant \varepsilon}$ be a Steklov isospectral deformation of the metric $g$ s.t there exists a family of smooth functions $\left( c_{s}:M\mapsto \mathbb{R^{+}_{*}}\right)_{0\leqslant s\leqslant \varepsilon} $ satisfying for any $s\in\left[ 0,\varepsilon\right]$, $$g_{s}=c_{s}g\hspace{0.1cm}.$$
Then the boundary deformation is flat at $s=0$ : $$\forall k\in\mathbb{N}^{*}\hspace{0.1cm},\hspace{0.1cm}\restriction{\partial^{k}_{s}c_{s}}{s=0}=0\hspace{0.2cm}\text{in $\partial M$.}$$
In particular, if $s\mapsto c_{s}$ is real-analytic, then for all $s\in\left[ 0,\varepsilon\right]$,
$$g_{s}^{\circ}=g^{\circ}\hspace{0.1cm}.$$    
\end{prop}
\begin{proof}
First, \textbf{Proposition \ref{Prop II.7}} combined with the injectivity of the X-ray transform on functions for closed Anosov manifolds (\textbf{Theorem \ref{Th II.5}}) gives $\restriction{\partial_{s}c_{s}}{s=0}=0$. Then, as already mentioned, the proof is done by induction on the order of derivation $k$ of the length maps $s\mapsto\mathcal{L}_{g_{s}^{\circ}}(\left[\gamma\right])$, $\left[\gamma\right]\in\mathcal{C}$. At each step, we combine \eqref{eqind} with the injectivity of the X-ray transform on functions.
\end{proof}

As mentioned above, under certain stronger assumptions, mainly about curvature, we can formulate a first result in the non-linear case by using Weyl's law (see \textbf{Proposition \ref{Prop II.1}}) and adapting arguments\footnote{Arguments originally put forward by A.El Soufi and S.Ilias in their work on conformal volume.} from Besson, Gallot and Courtois initially used to establish a similar spectral rigidity result for the Laplacian \cite[Proposition B.2]{BGC}. However, in our case, the statement is weaker since the total scalar curvature (i.e. the integral of scalar curvature) is not an invariant of the Steklov spectrum : 

\begin{Theo}\label{Th.IV.1}
    Let $(M,g)$ be a smooth compact Riemannian manifold of dimension $n\geq 4$ with boundary such that $(\partial M,g^{\circ})$ has a constant negative scalar curvature.
    \\Let $c:M\rightarrow \mathbb{R}^{+}_{*}$ be a smooth function satisfying 
    $$\mathrm{spec}(\Lambda_{cg})=\mathrm{spec}(\Lambda_{g})\hspace{0.1cm}.$$
If the total scalar curvatures at the boundary coincide 
    $$S_{(cg)^{\circ}}=S_{g^{\circ}}\hspace{0.1cm},$$ 
then 
    $$\restriction{c}{\partial M} \equiv 1\hspace{0.1cm}.$$
\end{Theo}
\begin{proof}
The proof follows exactly that of \cite[Proposition B.2]{BGC}.
    \\Let us write $h=cg$ and $c=e^{-2u}$ with $u$ a smooth real-valued function on $M$.
    \\The well-known expression of the scalar curvature after conformal change gives on $\partial M$ : 
    
    $$scal(g^{\circ})=e^{-2u}\left(scal(h^{\circ})-2(n-2)\Delta_{h^{\circ}}u-(n-3)(n-2)\lvert du\rvert_{h^{\circ}}^{2}\right)\hspace{0.1cm},$$ 
    
    \medskip\noindent \\then by integrating with respect to the canonical measure of $(\partial M,h^{\circ})$ : 
    
    $$\int_{\partial M}e^{2u}\left(-scal(g^{\circ})\right)dv_{h^{\circ}}=\int_{\partial M}\left(-scal(h^{\circ})\right)dv_{h^{\circ}}+(n-3)(n-2)\int_{\partial M}\lvert du\rvert_{h^{\circ}}^{2}dv_{h^{\circ}}\hspace{0.1cm}.$$ 
    
    \medskip\noindent Indeed, since the integral of the divergence of a vector field on a closed manifold is zero (by the classical Stokes' theorem), we have  
    
    $$\int_{\partial M}\left(\Delta_{h^{\circ}}u\right) dv_{h^{\circ}}=\int_{\partial M}\mathop{\rm div_{h^{\circ}}}(\nabla_{h^{\circ}}u)dv_{h^{\circ}}=0\hspace{0.1cm}.$$
    
    \medskip\noindent Thus, as the scalar curvature of $(\partial M,g^{\circ})$ is constant and by the Hölder's inequality applied with exponents $p=\dfrac{n-1}{2}$ and $q=\dfrac{n-1}{n-3}$ we get : 
    
    $$\begin{aligned}[t]\int_{\partial M}\left(-scal(h^{\circ})\right)dv_{h^{\circ}} &\leq \int_{\partial M}e^{2u}\left(-scal(g^{\circ})\right)dv_{h^{\circ}}\\
    &\leq -scal(g^{\circ})\left(\underbrace{\int_{\partial M}e^{\left(n-1\right)u}dv_{h^{\circ}}}_{=Vol_{g^{\circ}}(\partial M)}\right)^{\frac{1}{p}}Vol_{h^{\circ}}(\partial M)^{\frac{1}{q}}\hspace{0.1cm}.\end{aligned}$$
    \\Since $g$ and $h$ are Steklov isospectral and the volume of the boundary $\partial M$ is an invariant for the Steklov spectrum (\textbf{Corollary \ref{Coro II.2}}), we have $Vol_{g^{\circ}}(\partial M)=Vol_{h^{\circ}}(\partial M)$ and consequently the previous inequality becomes :
    
    $$\begin{aligned}[t]\int_{\partial M}\left(-scal(h^{\circ})\right)dv_{h^{\circ}}&\leq -scal(g^{\circ})\left(Vol_{g^{\circ}}(\partial M)\right)^{\frac{1}{p}}Vol_{h^{\circ}}(\partial M)^{\frac{1}{q}}\\
    &\leq -scal(g^{\circ})Vol_{g^{\circ}}(\partial M)\hspace{0.1cm},\end{aligned}$$ 
    
    \medskip\noindent obviously the equality holds iff $\restriction{u}{\partial M}=0$ . \footnote{Recall that for $f$ and $g$ two integrable functions and $p$, $q$ conjugate exponents the equality in Hölder's inequality holds iff $$\lvert g\rvert^{q}=\lambda^{p}\lvert f\rvert^{p}\hspace{0.1cm}\text{almost everywhere,}$$ with $\lambda=\dfrac{\lVert g\rVert_{q}^{q/p}}{\lVert f\rVert_{p}}$. Here we have $f=e^{2u}$ and $g\equiv1$ on $\partial M$ so we get that the equality holds iff $$e^{\left(n-1\right)u}\equiv Vol_{g^{\circ}}(\partial M)/Vol_{h^{\circ}}(\partial M)=1\hspace{0.1cm}\text{on $\partial M$,}$$ iff $$\restriction{u}{\partial M}=0\hspace{0.1cm}.$$}
    \medskip\noindent\\By assumption, we do have 
    $$\int_{\partial M}\left(-scal(h^{\circ})\right)dv_{h^{\circ}}=\int_{\partial M}\left(-scal(g^{\circ})\right)dv_{g^{\circ}}=-scal(g^{\circ})Vol_{g^{\circ}}(\partial M)\hspace{0.1cm},$$ 
    
    \medskip\noindent so the equality in the inequality and therefore $\restriction{u}{\partial M}=0$ .
\end{proof}
\medskip\noindent In dimension $3$, the assumption on total scalar curvatures at the boundary is already verified which gives in particular this elementary statement :

\begin{prop}\label{Prop III.5}
    Let $\left( M,g\right)$ be a smooth compact Riemannian $3$-manifold with boundary such that $(\partial M,g^{\circ})$ has constant sectional curvature.
    \\Let $h:=cg$ be a metric conformal to $g$ and assume that $h^{\circ}$ has constant sectional curvature.\footnote{By the uniformization theorem it is always possible to find such metrics.} \\If 
    $$\mathrm{spec}(\Lambda_{cg})=\mathrm{spec}(\Lambda_{g})\hspace{0.1cm},$$
then
$$scal(h^{\circ})=scal(g^{\circ})\hspace{0.1cm}.$$

    \medskip\noindent In addition, if this scalar curvature is negative,
    $$\restriction{c}{\partial M} \equiv 1\hspace{0.1cm}.$$
\end{prop}
\begin{proof}
As usual let us note $c=e^{2u}$ with $u$ a smooth real-valued function on $M$. The expression of the scalar curvature after conformal change in dimension $2$ leads to :
$$\int_{\partial M}scal(h^{\circ})dv_{h^{\circ}}=\int_{\partial M}e^{2u}scal(g^{\circ})dv_{h^{\circ}}=scal(g^{\circ})Vol_{g^{\circ}}(\partial M)\hspace{0.1cm}.$$

\medskip\noindent By isospectrality, $Vol_{h^{\circ}}(\partial M)=Vol_{g^{\circ}}(\partial M)$, and since $h^{\circ}$ has constant curvature we obviously deduce that 
$$scal(h^{\circ})=scal(g^{\circ})\hspace{0.1cm}.$$
Therefore, $\varphi:=2\restriction{u}{\partial M}$ satisfies the equation 
$$-\Delta_{h^{\circ}}\varphi +\kappa=\kappa e^{\varphi}\hspace{0.1cm},$$
where $\kappa:=scal(g^{\circ})$. But if $\kappa$ is negative, it is easy to see that $\varphi\equiv 0$ is the unique solution and, equivalently, that there is one and only one metric with scalar curvature $\kappa$ in the conformal class of $g^{\circ}$. We refer, for example, to the proof of \cite[Theorem 6.8]{Aub} and more generally to T.Aubin's work on the prescription of scalar curvature.
\end{proof}

With the same idea of using volume as a spectral invariant, we could formulate a second type of result for the non-linear problem, based on the relationship established in \cite[Theorem 1.2]{CrDai} (via arguments due to Katok \cite{Kat}) between the (marked) lengths of closed geodesics and volume, but the statement is too unsatisfactory to appear here. In particular, the assumption appearing in \cite[Theorem 1.2]{CrDai} is a monotonicity condition on the marked length spectrum, and it is not clear how often this condition is true.

\section{A suitable expression of the DN map}\label{App B}

Here we provide a useful identity for the calculation of the subprincipal symbol performed in \textbf{Lemma \ref{Prop III.12}}. This easy-to-establish identity is undoubtedly very well known, but the author has found no clear reference for it, so it is reproduced here. 

\begin{Lem}\label{lem1.2.1}
Let $c:M\rightarrow \mathbb{R}^{+}_{*}$ be a smooth function satisfying 
$\restriction{c}{\partial M} \equiv 1$. \\Then 
$$\Lambda_{cg}=\Lambda_{g,-q_{c}}-\dfrac{n-2}{4}\left(\restriction{\left(\partial_{\nu}c\right)}{\partial M}\right)Id\hspace{0.1cm},$$ where 
$$q_{c}=c^{\frac{n-2}{4}}\Delta_{cg}(c^{-\frac{n-2}{4}})\hspace{0.1cm}.$$
\end{Lem}

\begin{proof}
It is well-known that the Laplace-Beltrami operator transforms under conformal scaling of the metric by 
$$\Delta_{cg}=c^{-\frac{n+2}{4}}\left(\Delta_{g}+q_{c}\right)c^{\frac{n-2}{4}}\hspace{0.1cm}.$$ 

\medskip\noindent Thus, for any $f\in H^{1/2}(\partial M)$, $u$ is solution in $H^{1}(M)$ of 
$$\left\{\begin{aligned}
 -\Delta_{cg}u &= 0 \\
\restriction{u}{\partial M} &= f \\
\end{aligned}\right.$$
iff 
$$v:=c^{\frac{n-2}{4}}u$$ 

\medskip\noindent is solution in $H^{1}(M)$ of the following Dirichlet problem 

$$\left\{\begin{aligned}
 \mathcal{L}_{g,-q_{c}}v &= 0 \\
\restriction{v}{\partial M} &= f \\
\end{aligned}\right.\hspace{0.5cm},$$

\medskip\noindent where $\mathcal{L}_{g,-q_{c}}:=-\Delta_{g}-q_{c}$\hspace{0.1cm}.
Therefore, 
$$\begin{aligned}
    \Lambda_{cg}f:=\restriction{\partial_{\nu}u}{\partial M}&=\restriction{\partial_{\nu}(c^{-\frac{n-2}{4}}v)}{\partial M} \\ &= -\frac{n-2}{4}\restriction{\left((c^{-\frac{n+2}{4}}\partial_{\nu}c)v\right)}{\partial M} + \restriction{\left(c^{-\frac{n-2}{4}}\partial_{\nu}v\right)}{\partial M} \\
\end{aligned}$$
and since $\restriction{v}{\partial M}=f$,
$$\begin{aligned}
  \Lambda_{cg}f &=-\frac{n-2}{4}\restriction{\left((\partial_{\nu}c)v\right)}{\partial M} + \restriction{\partial_{\nu}v}{\partial M} \\
  &=  -\frac{n-2}{4}\restriction{\left(\partial_{\nu}c\right)}{\partial M}f + \Lambda_{g,-q_{c}}f\hspace{0.1cm}.\\
\end{aligned}$$
\end{proof}

\begin{Rq}\label{Rk III.9}
If $M$ is of dimension $2$, we easily deduce the well-known conformal invariance of the DN map i.e 
$$\Lambda_{cg}=\Lambda_{g}\hspace{0.1cm},$$ if 
$$\restriction{c}{\partial M} \equiv 1\hspace{0.1cm}.$$
\end{Rq}

\medskip\noindent The above formula can be refined with the :

\begin{prop}\label{Prop B.3}
    Let $q\in C^{\infty}(M)$ such that $0$ is not a Dirichlet eigenvalue of $-\Delta_{g}+q$ so that $$\Lambda_{g,q}:H^{1/2}(\partial M)\rightarrow H^{-1/2}(\partial M)$$ is well defined. \\Then 
    $$\Lambda_{g,q}=\Lambda_{g}+R_{-1}\hspace{0.3cm},$$ 
    
    \medskip\noindent where $R_{-1}\in\Psi^{-1}(\partial M)$ has for principal symbol 
    
    $$\sigma_{R_{-1}}(x,\xi)=\frac{q^{\circ}(x)}{2\lvert\xi\rvert_{g^{\circ}}}\hspace{0.3cm}.$$
\end{prop}
\begin{proof}
This result is a direct consequence of the expressions in boundary normal coordinates for the full symbols of the pseudodifferential operators $\Lambda_{g,q}$ and $\Lambda_{g}$ computed, for example, in \cite[Lemma 8.6]{DDSF} and \cite[Section 1]{LU} respectively.   
\end{proof}

\medskip\noindent In particular, we deduce the following :

\begin{cor}\label{coro 1.3.1}
Let $c:M\rightarrow \mathbb{R}^{+}_{*}$ be a smooth function satisfying $\restriction{c}{\partial M} \equiv 1$. \\Then 
$$\Lambda_{cg}=\Lambda_{g}+\alpha_{n}\left(\restriction{\left(\partial_{\nu}c\right)}{\partial M}\right)Id + R_{-1}\hspace{0.3cm},$$ 

\medskip\noindent where $\alpha_{n}=-\dfrac{n-2}{4}$ and $R_{-1}\in\Psi^{-1}(\partial M)$.
\end{cor}

\bibliographystyle{alphaurl}
\bibliography{bibl}

\end{document}